\documentclass[a4paper, reqno]{amsart}
\usepackage{amsmath, amssymb, eucal, amscd, amstext, enumerate, mathrsfs, yfonts, amsfonts, verbatim, tikz, relsize}

\usepackage[plainpages=false,colorlinks,hyperindex,pdfpagemode=None,bookmarksopen,linkcolor=red,citecolor=blue,urlcolor=blue]{hyperref}
\usepackage{pdflscape}

\newtheorem{thm}{Theorem}[section]
\newtheorem{lem}[thm]{Lemma}
\newtheorem{eg}[thm]{Example}
\newtheorem{prop}[thm]{Proposition}
\newtheorem{prop-defn}[thm]{Proposition and Definition}
\newtheorem{cor}[thm]{Corollary}
\newtheorem{defn}[thm]{Definition}
\newtheorem{rem}[thm]{Remark}

\numberwithin{equation}{section}

\newcommand{\ti}{\tilde}
\newcommand{\smnoind}{\smallskip\noindent}
\newcommand{\la}{\langle}
\newcommand{\ra}{\rangle}

\newcommand{\BR}{\mathbb{R}}
\newcommand{\RP}{\mathbb{R}^+}
\newcommand{\BN}{\mathbb{N}}
\newcommand{\BC}{\mathbb{C}}

\newcommand{\CL}{\mathcal{L}}
\newcommand{\CS}{\mathcal{S}}
\newcommand{\CN}{\mathcal{N}}

\newcommand{\KI}{\mathfrak{I}}
\newcommand{\KH}{\mathfrak{H}}
\newcommand{\KK}{\mathfrak{K}}
\newcommand{\sa}{\mathrm{sa}}

\newcommand{\cb}{\mathrm{cb}}
\newcommand{\CB}{\mathrm{CB}}
\newcommand{\CPB}{\mathrm{CB}^+}

\newcommand{\WQ}{\mathcal{WQ}}
\newcommand{\WS}{\mathcal{WS}}
\newcommand{\Ucp}{\mathrm{UCP}}
\newcommand{\Ucpw}{\mathrm{UCP}_w}
\newcommand{\Ccpw}{\mathrm{CC}^+_w}

\newcommand{\Pc}{\mathrm{Pc}}
\newcommand{\Pcw}{\mathrm{Pc}_w}

\newcommand{\rd}{\mathrm{d}}
\newcommand{\Mor}{\CB^+}
\newcommand{\Morw}{\CB^+_w}
\newcommand{\Morc}{\mathrm{CC}^+}
\newcommand{\rn}{\lambda}
\newcommand{\reg}{\mathrm{reg}}
\newcommand{\wu}{\mathrm{wu}}

\begin{document}

\title{Dual spaces of operator systems}

\author{Chi-Keung Ng}

\address[Chi-Keung Ng]{Chern Institute of Mathematics and LPMC, Nankai University, Tianjin 300071, China.}
\email{ckng@nankai.edu.cn}

\date{\today}


\begin{abstract}
The aim of this article is to give an infinite dimensional analogue of a result of Choi and Effros concerning dual spaces of finite dimensional unital operator systems.  

An %changed
(not necessarily unital) \emph{operator system} is a self-adjoint subspace of $\mathcal{L}(\mathfrak{H})$, equipped with the induced matrix norm and the induced matrix cone. 
We say that an operator system $T$ is \emph{dualizable} if one can find an equivalent dual matrix norm on the dual space $T^*$ such that under this dual matrix norm and the canonical dual matrix cone, $T^*$ becomes a dual operator system. 

We show that \textcolor{magenta}{a complete} operator system $T$ is dualizable if and only if the ordered normed space $M_\infty(T)^\mathrm{sa}$ satisfies a form of bounded decomposition property. 
In this case, 
$$\|f\|^\mathrm{d}:= \sup \big\{\big\|[f_{i,j}(x_{k,l})]\big\|: x\in M_n(T)^+; \|x\|\leq 1; n\in \mathbb{N}\big\} \quad (f\in M_m(T^*); m\in \mathbb{N}),$$
is the largest dual matrix norm that is equivalent to and dominated by the original dual matrix norm on $T^*$ that turns it into a dual operator system, denoted by $T^\mathrm{d}$.
It can be shown that $T^\mathrm{d}$ is again dualizable.

For every completely positive completely bounded map $\phi:S\to T$ between dualizable operator systems, there is a unique weak-$^*$-continuous completely positive completely  bounded map $\phi^\mathrm{d}:T^\mathrm{d} \to S^\mathrm{d}$ which is compatible with the dual map $\phi^*$. 
From this,  we obtain a full and faithful functor from the category of dualizable operator systems to that of dualizable dual operator systems.

Moreover, we will verify that that if $S$ is either a $C^*$-algebra or a unital \textcolor{magenta}{complete} operator system, then $S$ is dualizable and the canonical weak-$^*$-homeomorphism from the unital complete operator system $S^{**}$ to the operator system $(S^\mathrm{d})^\mathrm{d}$ is a completely isometric complete order isomorphism. 

Furthermore, via the duality functor above, the category of $C^*$-algebras and that of unital \textcolor{magenta}{complete} operator systems (both equipped with completely positive complete contractions as their morphisms) can be regarded as full subcategories of the category of dual operator systems (with weak-$^*$-continuous completely positive complete contractions as 
its morphisms). 

Consequently, a nice duality framework for operator systems is obtained, which includes all $C^*$-algebras  and all unital \textcolor{magenta}{complete}  operator systems.
\end{abstract}

\maketitle

\section{Introduction}

\medskip

%v2
Operator systems is an important subject in operator algebras. 
This subject recently gained more attention because of the study of ``spectral truncation'', as introduced by Connes and van Suijlekom in \cite{CvD}. 
In particular, finite dimensional unital operator systems of Toeplitz matrices were considered in \cite{CvD}, and ``operator system duals'' of them (coming from a result of Choi and Effros; see \cite[Corollary 4.5]{CE}) were studied (see \cite{Far} for further developments on this direction). 

\medskip

Another interest on the duality of finite dimensional unital operator systems comes from its connection with quantum information theory. 
In particular, the notion of ``graph operator systems'' was introduced in \cite{DSW} to obtain quantum versions of some theories of Shannon. 
Moreover, ``dual graph operator systems'' for finite graphs, which are ``operator system duals'' (defined using the result of Choi and Effros in \cite[Corollary 4.5]{CE}) of graph systems were studied in \cite[\S 3.5]{GMS}. 

\medskip

%v2
Since infinite dimensional operator systems are also drawing some attention in quantum information theory (see e.g. \cite{CJ, Yashin}), it seems that an analogue of \cite[Corollary 4.5]{CE} in the infinite dimensional case could be a useful tool. 
This is one motivation of this article.

\medskip

The second motivation of this article is to set up  a nice duality framework of infinite dimensional operator systems, so that it can be served as a basis for our possible further study on dualities for tensor products of operator systems (see basic theory for tensor products of operator systems can be found in, e.g., \cite{KPTT,LN}). 

\medskip

Note that there has been a study on duality of infinite dimensional operator systems in \cite{Han}, but the concern of the current article is different from \cite{Han}, in the sense that we are interested in the case when the dual space is a \emph{dual} quasi-operator system (see Definition \ref{defn:dual-quasi-op-sys}) instead of a quasi-operator system as in \cite{Han}. 

\medskip

Notice that, in the infinite dimensional case, the cone of the dual space of a unital operator system may not admits an order unit; for example, the cone of the dual space of the $C^*$-algebras of all convergent sequences does not have an order unit (see Example \ref{eg:SMOS}(c)). 
This forces us to consider \emph{(not necessarily unital) operator system}, i.e. self-adjoint subspaces of $\CL(\KH)$ that may not contain the identity element. 
In this case, we need to take into account both the induced matrix cones and the induced matrix norms on those spaces. 
Note that in the case of a unital operator system, the induced matrix cone and the induced matrix norm determine each other, via the fixed order unit, but it is not the case in the non-unital case (see e.g., Example \ref{eg:SMOS}(b)). 

\medskip 

Observe also that the dual space $S^*$ of a unital operator system $S$ is, in general, not  an operator system in the strict sense (i.e., under the dual matrix cone and the dual matrix norm, $S^*$ may not be a subspace of some $\CL(\KH)$); for example, the dual space of $M_2(\BC)$ is not an operator system under the dual operator space structure (see e.g. \cite[Corollary 2.8]{Ng-MOS}). 
Therefore, our question is whether $S$ is ``dualizable'', in the sense that $S^*$ admits a suitable dual matrix norm equivalent  to the original one, so that $S^*$ becomes an operator system under the canonical dual matrix cone (in the finite dimensional case, \cite[Corollary 4.5]{CE} tells us that $S^*$ always admits a matrix order unit turning into a unital operator system). 
Moreover, if one wants to study dualities for operator system tensor products, one actually needs a canonical way to turn the dual space of an operator system into an operator system.

\medskip

In this article, we will show that if $S$ is a unital \textcolor{magenta}{complete} operator system, then $S$ is dualizable (see Theorem \ref{thm:dual-oper-sys}(a)). 
More precisely, using the weak-$^*$-version of a construction employed in \cite{Ng-MOS} for the study of unitization of quasi-operator systems, we found that  (see Theorem \ref{thm:quasi-oper-sys}(b)) 
$$\|[g_{i,j}]\|^\rd:= \sup \big\{\big\|[g_{i,j}(y_{k,l})]\big\|: [y_{k,l}]\in M_n(S)^+; \|[y_{k,l}]\|\leq 1; n\in \BN\big\} \qquad (m\in \BN; [g_{i,j}]\in M_m(S^*))$$
is an equivalent dual matrix norm on $S^*$ that turns $S^*$ into an operator system. 
In fact, it is the largest equivalent dual matrix norm dominated by the original dual matrix norm on $S^*$ that does the job (see Theorem \ref{thm:quasi-oper-sys}(b) and Theorem \ref{thm:dual-oper-sys}(a)).
In this case, the resulting (not necessarily unital) operator system, denoted by $S^\rd$, can be regarded as ``the operator system dual'' of the unital operator system $S$.

\medskip

After achieving this, a natural question is whether the dual space of $S^\rd$ can again be turned into an operator system. 
If it is the case, one also wants to know whether $(S^\rd)^\rd$ coincides $S^{**}$ as operator systems.

\medskip

Unfortunately it may not be possible to turn the dual space of a non-unital operator system into an operator system in general (see Example \ref{eg:SMOS}(a)).  
This draws us to the study of ``dualizability'' of non-unital operator systems. 
We eventually find that \textcolor{magenta}{a complete} operator system $T$ is ``dualizable'' if and only if one can find $r>0$ such that every self-adjoint element in $\bigcup_{n\in \BN} M_n(T)$ with norm one is a difference of two elements in the matrix cone with their norms dominated by $r$ (Theorem \ref{thm:quasi-oper-sys}(a)). 
Furthermore, if $T$ is dualizable, then so is $T^\rd$ (Proposition \ref{prop:predual}(b)). 
In the case of a unital \textcolor{magenta}{complete} operator system $S$, we also show that the unital operator system $S^{**}$ is the same as the operator system $(S^\rd)^\rd$, in the sense that the canonical complete order isomorphism $\tau_S:S^{**}\to (S^\rd)^\rd$ is completely isometric (Theorem \ref{thm:dual-oper-sys}(b)).  

\medskip 

If $S,T$ are two dualizable \textcolor{magenta}{complete} operator systems and $\phi:S\to T$ is a completely positive complete bounded map, then we can find a unique weak-$^*$-continuous completely positive complete bounded map $\phi^\rd:T^\rd\to S^\rd$ compatible with the dual map $\phi^*$. 
Through this, we obtain a full and faithful contravariant functor from the category of dualizable operator systems to the category of dualizable dual operator systems. 
Moreover, the subcategory consisting of dualizable operator systems $S$ with $\tau_S$ being completely isometric can be identified, via this contravariant functor, with a full subcategory of the category of dualizable dual operator systems (see Theorem \ref{thm:dualization-functor}). 

\medskip

On the other hand, we show that all (not necessarily unital) $C^*$-algebras are dualizable (see Theorem \ref{thm:dual-oper-sys}(a)).
Furthermore, for a $C^*$-algebra $A$, the complete order isomorphism $\tau_A$ is also completely isometric (see Theorem \ref{thm:dual-oper-sys}(b)).

\medskip

This sets up a framework for duality of operator systems that includes both $C^*$-algebras and unital \textcolor{magenta}{complete} operator systems. 

\medskip

In the appendix, we will present two remarks. 
The first one is that, if one insists to have a unital operator system as the ``dual object'' of a unital operator system $S$, then one may consider the ``weak-$^*$-unitization'', $S^\diamond$,  of $S^\rd$ (see Proposition \ref{prop:alt-def-dual-MOS} and the discussion following it), which can be regarded as a unital dual operator system equipped with a fixed character. 
This gives a nice ``duality functor''  as well (see Corollary \ref{cor:unital-dual}). 
Note that one can recover the dual space of $S$, up to equivalent dual matrix norms, from the kernel of the fixed character on $S^\diamond$ (see the discussion following Corollary \ref{cor:unital-dual}).

\medskip

The second remark concerns with analogues of results in the main context concerning the ``non-matrix case'' (i.e. concerning with ordered normed spaces and functions systems). 
Readers who are not familiar with operator spaces and operator systems may have a look at Subsection \ref{subsec:non-matrix} first, in order to gain some ideas about the results and definitions in the main context (which are basically matrix versions of the materials in Subsection \ref{subsec:non-matrix}). 

\medskip

\section{Prelimaries: MOS and Quasi-operator systems}

\medskip

In this section, we will give some notations and definitions. 
We will also present some elementary results and known facts. 

\medskip

For a complex vector space $X$ and a positive integer $n$, we identify $M_n(X)\subseteq M_{n+1}(X)$ by putting elements of $M_n(X)$ into the upper left corner of $M_{n+1}(X)$. 
Under this identification, we denote 
$$M_\infty(X) := {\bigcup}_{n\in \BN} M_n(X).$$ 
In this case, $M_\infty(X)$ is a bimodule over $M_\infty(\BC)$. 
If $Y$ is another vector space and $\varphi:X\to Y$ is a complex linear map, we define a map  $\varphi^{(\infty)}:M_\infty(X)\to M_\infty(Y)$ by 
$$\varphi^{(\infty)}\big([x_{i,j}]\big) := \big[\varphi(x_{i,j})\big] \qquad ([x_{i,j}]\in M_\infty(X)).$$
For $n\in \BN$, we set $\varphi^{(n)} := \varphi^{(\infty)}|_{M_n(X)}$. 

\medskip

Recall that a complex vector space $X$ is an \emph{operator space} if there is a norm $\|\cdot\|$ on $M_\infty(X)$ (called the \emph{matrix norm}) such that for every $n\in \BN$, $a_1,\dots,a_n, b_1, \dots, b_n\in M_\infty(\BC)$ and $x_1,\dots ,x_n\in M_\infty(X)$, 
$$\Big\|{\sum}_{k=1}^n  a_k^*x_kb_k\Big\| \leq \Big\|{\sum}_{k=1}^n  a_k^*a_k\Big\|^{1/2}\Big\|{\sum}_{k=1}^n  b_k^*b_k\Big\|^{1/2} \max_{k=1,\dots, n}\|x_n\|.$$
The above is a reformulation of the original definition in \cite{ER}, and can be found in, e.g., \cite[Theorem 2.1]{Ng-reg-mod} (note that although the completeness assumption is stated in \cite[Theorem 2.1]{Ng-reg-mod}, the statement itself is true without the completeness assumption with the same argument). 

\medskip

For complex normed spaces $X$ and $Y$, we denote by $\CL(X;Y)$ the set of bounded linear maps from $X$ to $Y$, and put $\CL(X):= \CL(X;X)$. 
Let us also recall that a linear map $T:X\to Y$ is \emph{bounded below} if 
$${\inf}_{\|x\|=1} \|Tx\| > 0.$$  

\medskip

For two operator spaces $X$ and $Y$, a complex linear map $\varphi:X \to Y$ is  
said to be \emph{completely bounded} %(respectively, \emph{completely contractive}) 
if $\varphi^{(\infty)}\in \CL\big(M_\infty(X); M_\infty(Y)\big)$. %(respectively, contractive). 
Moreover, $\varphi$ is called a 
\begin{itemize}
	\item \emph{complete isometry} if $\varphi^{(\infty)}$ is isometric;
	
	\item \emph{complete contraction} if $\varphi^{(\infty)}$ is contractive;

	\item \emph{complete embedding} if $\varphi^{(\infty)}$ is bounded and bounded below.
\end{itemize}

\medskip

For our study of duality of operator system, the natural framework to work on, is the notion of matrix ordered operator spaces (MOS), as introduced by Werner in \cite{Wern1}, as well as its dual version (see the next section). 
However, since the dual space of an operator system may not be a MOS, we will go a bit more general to the setting of SMOS. 

\medskip

\begin{defn}\label{defn:SMOS}
(a) An operator space $X$ is called a \emph{semi-matrix ordered operator space} (\emph{SMOS}) if there exist an involution $^*: X\to X$ with the induced map $^*:M_\infty(X)\to M_\infty(X)$ (sending $[x_{i,j}]$ to $[x_{j,i}^*]$) being isometric, and a 
norm-closed  subset 
\begin{equation*}
M_\infty(X)^+\subseteq M_\infty(X)^\sa := \{x\in M_\infty(X): x^* = x \}
\end{equation*}
(known as the \emph{matrix cone}) satisfying 
\begin{equation}\label{eqt:mat-cone}
{\sum}_{k=1}^n  a_k^*x_ka_k\in M_\infty(X)^+ \qquad (n\in \BN; a_1,...,a_n\in M_\infty; x_1,...,x_n\in M_\infty(X)^+).
\end{equation}
In this case, we denote $M_n(X)^\sa := M_n(X)\cap M_\infty(X)^\sa$ and $M_n(X)^+ := M_n(X)\cap M_\infty(X)^+$. 

\smnoind
(b) (\cite{Wern1})  A SMOS $X$ is called a \emph{matrix ordered operator space} (\emph{MOS}) if the matrix cone is \emph{proper}; i.e., 
$M_\infty(X)^+ \cap - M_\infty(X)^+ = \{0\}$.
\end{defn}

\medskip

Observe that a subset $M_\infty(X)^+$ satisfying Relation \eqref{eqt:mat-cone} is norm-closed in $M_\infty(X)^\sa$ if and only if $M_n(X)^+$ is norm-closed in $M_n(X)^\sa$ for all $n\in \BN$.

%\medskip

%For any Hilbert space $\KH$, the vector space $\CL(\KH)$ admits a canonical MOS structure. 
%We equip with $M_\infty$ with the MOS structure induced from the canonical inclusion $M_\infty \subseteq \CL(\ell_2)$. 

\medskip

Consider two SMOS $X$ and $Y$. 
If $\varphi:X\to Y$ is a complex linear map satisfying $\varphi(x^*) = \varphi(x)^*$ ($x\in X$), then  $\varphi$ is said to be 
\begin{itemize}
	\item \emph{completely positive} if $\varphi^{(\infty)}(M_\infty(X)^+)\subseteq M_\infty(Y)^+$;
	
	\item \emph{completely order monomorphic} if $\varphi$ is injective and for every $x\in M_\infty(X)^\sa$, the condition $x\in M_\infty(X)^+$ is equivalent to  $\varphi^{(\infty)}(x)\in M_\infty(Y)^+$; 

%	\item $\varphi$ is called a \emph{complete order isomorphism} if $\varphi$ is bijective and completely order monomorphic;
%v2
	\item \emph{a MOS isomorphism} if $\varphi$ is a bijective completely order monomorphic complete embedding. 
\end{itemize}
We denote by $\CPB(X,Y)$ the set of all completely bounded completely positive maps from $X$ to $Y$. 
Moreover, the set of completely positive complete contractions from $X$ to $Y$ will be denoted by $\Morc(X,Y)$. 

\medskip

By an \emph{operator system}, we mean a self-adjoint subspace $S$ of some $\CL(\KH)$, equipped with the induced MOS structure. 
Note that in the case when $S$ is \emph{unital}; i.e,  it contains the identity of $\CL(\KH)$, one only needs to consider the induced matrix cone on $S$, since the induced matrix norm can be recovered from the identity and the matrix cone. 
However, in the non-unital case, one needs to consider both the matrix cone and the matrix norm. 

\medskip

Let us recall the following definition from \cite{Ng-MOS} (in particular, see \cite[Theorem 2.6]{Ng-MOS}). 
Notice that our definition is slightly different from the one in \cite{Wern1}, in the sense that ``operator systems'' in \cite{Wern1} are precisely ``quasi-operator systems'' here. 

\medskip

\begin{defn}
Let $X$ be a MOS. 
If there is a MOS isomorphism $\Psi$ from  $X$ onto an operator system, then $X$ is called a \emph{quasi-operator system}. 
\end{defn}

\medskip

By rescaling, we may always assume that the map $\Psi$ in the above is a complete contraction.

\medskip

Closely related to the notation of quasi-operator systems is the notion of partial unitizations (as originally introduced in \cite{Wern-subsp}).
The following is a reformulation of this notion, and is taken from Definition 3.2 and Theorem 3.3 of \cite{Ng-MOS}. 

\medskip

\begin{prop}\label{prop:unital}
	If $X$ is a MOS, then there is a unique pair $(X^1, \iota_X)$ with $X^1$ being a unital operator system and $\iota_X\in \Morc(X; X^1)$ being completely order monomorphic such that for each unital operator system $S$ and $\Phi\in \Morc(X;S)$, there is a unique unital completely positive map $\ti \Phi:X^1\to S$ satisfying $\Phi = \ti \Phi \circ \iota_X$. 
	
	In this case, $(X^1, \iota_X)$ (or simple $X^1$) is called a \emph{partial unitization} of $X$. 
	If it happens that $\iota_X$ is a complete embedding, then $X^1$ is called a \emph{unitization} of $X$.
\end{prop}

\medskip

It is easily seen from the definition above that $X^1 = \iota_X(X) + \BC 1$. 
Note, however, that even if $\Phi:X\to S$ is a completely order monomorphic complete isometry, there is no guarantee that $\ti \Phi:X^1\to S$ is a complete isometry (see \cite[Proposition 2.13]{Wern1}). 

\medskip

The following result is a combination of Theorem 2.6 and Theorem 3.3(a) of \cite{Ng-MOS}.

\medskip

\begin{prop}\label{prop:unit-quasi-os}
Let $X$ be a MOS. 
Then $X$ is a quasi-operator system (respectively, an operator system) if and only if $\iota_X$ is a complete embedding (respectively, a complete isometry). 
\end{prop}

\medskip

\begin{lem}\label{lem:image-weak-st-cont-bdd-below}
	Let $E$ and $F$ be Banach spaces.
	Let $\Phi$ be a bounded and bounded below weak-$^*$-continuous linear map from the dual space, $E^*$, of $E$ to the dual space of $F$. 
	Then $\Phi(E^*)$ is a weak-$^*$-closed subspace of $F^*$, and $\Phi$ is a weak-$^*$-homeomorphism from $E^*$ onto $\Phi(E^*)$. 
\end{lem}

\medskip

In fact, let $\Psi:F\to E$ be the bounded linear map with $\Phi = \Psi^*$.
Denote $F_\Psi:= F/\ker \Psi$ and define $\Psi_0: F_\Psi \to E$ to be the bounded linear map induced by $\Psi$. 
Then  $F_\Psi^*$ can be identified, as Banach spaces,  with the weak-$^*$-closure of $\Phi(E^*)$ in $F^*$, in a canonical way. 
Under this identification, one has $\Psi_0^* = \Phi$. 
By its definition, $\Psi_0$ is injective. 
Moreover, as $\Phi$ is bounded below, so is $\Psi_0^*$.
This implies that $\Psi_0$ is a Banach space isomorphism.

\medskip

For a (real or complex) normed space $E$,  we  put 
$$B_E:=\{x\in E:\|x\|\leq 1 \}.$$ 
If $f = [f_{i,j}]\in M_\infty(E^*)$, then we use $\theta_f$ to denote the linear map from $E$ to $M_\infty(\BC)$ given by 
$$\theta_f(x) := [f_{i,j}(x)]\qquad (x\in E).$$  
It is well-known that for each $n\in \BN$, the assignment $\theta: f \mapsto \theta_f$ induces a linear bijection from $M_n(E^*)$ onto $\CL(E;M_n(\BC))$. 
We may sometimes ignore this identification and regard
$$\CL(E;M_n(\BC)) = M_n(E^*).$$
In the same way, one obtains a (not necessarily surjective) map $\theta: M_n(E) 
\to \CL(E^*;M_n(\BC))$ given by 
$$\theta_x(f) := [f(x_{i,j})]\qquad (f\in E^*),$$ 
when $x=[x_{i,j}]\in M_n(E)$. 

\medskip

%v2
\begin{lem}\label{lem:weak-st-cont}
$\varphi\in \CL(E^*;M_n(\BC))$ is weak-$^*$-continuous if and only if there exists $x\in M_n(E)$ such that $\varphi = \theta_x$. 
\end{lem}

\medskip

One may set up a duality between $M_\infty(E)$ and $M_\infty(E^*)$ by 
\begin{equation}\label{eqt:def-Mn-dual}
f(x) := {\sum}_{k,l=1}^\infty f_{k,l}(x_{k,l}) \qquad \big(f=[f_{k,l}]\in M_\infty(E^*), x=[x_{k,l}]\in M_\infty(E) \big).
\end{equation}
The weakest vector topology on $M_\infty(E^*)$ under which $f\mapsto f(x)$ is continuous  for every $x\in M_\infty(E)$ will be denoted by $\sigma\big(M_\infty(E^*),M_\infty(E)\big)$, 
and will be called the \emph{weak-$^*$-topology} on $M_\infty(E^*)$. 

\medskip
 
It is easy to see that 
a net $(f^{(\alpha)})_{\alpha\in \Lambda}$ in $M_\infty(E^*)$ converge to $f$ under the weak-$^*$-topology if and only if the net $(f^{(\alpha)}_{k,l})_{\alpha\in \Lambda}$  in $E^*$ weak-$^*$-converges to $f_{k,l}$, for every $k,l\in \BN$ (recall that there is only a finite number of non-zero terms in a matrix $[x_{k,l}]\in M_\infty(E)$). 

\medskip

When $n\in \BN$, the topology on on $M_n(E^*)$ induced by $\sigma\big(M_\infty(E^*),M_\infty(E)\big)$ will be denoted by $\sigma\big(M_n(E^*),M_n(E)\big)$ (according to Relation \eqref{eqt:def-Mn-dual}, the values of $x_{k,l}$ for $k>n$ or $l>n$ play no roles in this induced topology on $M_n(E^*)$). 

\medskip

For a SMOS $X$, the dual space $X^*$ is a SMOS under the involution given by $[f_{i,j}]^* := [g_{i,j}]$, where $g_{i,j} = (f_{j,i})^*$ ($i,j\in \BN$) for $[f_{i,j}]\in M_\infty(X^*)$, the matrix norm given by  
$$\|f\|_{M_\infty(X^*)} = \sup \big\{ \big\|\theta_f^{(\infty)}(x)\big\|: x\in B_{M_\infty(X)}\big\} \qquad (f\in M_\infty(X^*))$$
(see, e.g., \cite[p.41]{ER}), as well as the matrix cone given by: 
$$M_\infty(X^*)^+\ :=\ \big\{f\in M_\infty(X^*)^\sa: \theta_f\in \CPB(X;M_\infty(\BC))\big\}. $$ 
Observe that as $\theta_f$ is automatically completely bounded when $f\in M_n(X^*)$ (because $\theta_f$ is a bounded linear map from $X$ to $M_n(\BC)$), the condition $\theta_f\in \CPB(X;M_\infty)$ means that $\theta_f$ is completely positive. 

\medskip

\emph{Whenever we use the notation $X^*$, we means the dual space $X^*$ is equipped with the above SMOS structure. }

%\medskip

%\begin{eg}
%	Let $A$ be a $C^*$-algebra. 
%	It is not hard to check that the bidual SMOS structure on $A^{**}$; i.e., the dual SMOS structure of the MOS $A^*$,  coincides with the dual operator system structure coming from the von Neumann algebra structure on $A^{**}$. 
%\end{eg}

\medskip

In general, $X^*$ may not be a MOS even when $X$ is an operator system (see Example \ref{eg:SMOS}(a)). 
It is also possible that $X^*$ is a MOS, but is not a quasi-operator system  (compare  Proposition \ref{prop:dual-sp-MOS} and Theorem \ref{thm:quasi-oper-sys}(a)). 

\medskip 

Let us end this section with some well-known facts. 
The first one can be obtained by a combination of arguments of Proposition 5.4.1 and Lemma 5.4.3 of \cite{ER} (see also \cite{Paul}). 
Observe that weak-$^*$-continuous complete contractions from $A^{**}$ to $M_n(\BC)$ are precisely complete contractions from $A$ to $M_n(\BC)$. 

\medskip

\begin{lem}\label{lem:decomp}
	Let $A$ be a $C^*$-algebra and $n\in \BN$. 
	Suppose that $\phi:A\to M_n(\BC)$ is a complete contraction (respectively, $\phi:A^{**}\to M_n(\BC)$ is a weak-$^*$-continuous complete contraction). 
	There exist $\phi_1, \phi_2, \phi_3, \phi_4\in \Morc(A;M_n(\BC))$
	(respectively, weak-$^*$-continuous maps $\phi_1, \phi_2, \phi_3, \phi_4\in \Morc(A^{**};M_n(\BC))$) such that $\phi = \phi_1 - \phi_2 + \mathrm{i} \phi_3 - \mathrm{i}\phi_4$. 
\end{lem}

\medskip

The second well-known fact is a weak-$^*$-continuous version of Arveson's extension theorem.  
Let us first recall that if $N$ is a closed subspace of an operator space $X$.
Then we can equip the quotient space $X/N$ with a canonical matrix norm as follows: 
$$\|y\|:= \inf \{\|x\|: x\in M_\infty(X); q^{(\infty)}(x) = y \} \qquad (y\in M_\infty(X/N)),$$
where $q:X\to X/N$ is the quotient map. 
In this case, $X/N$ is an operator space, and is called a \emph{quotient operator space} of $X$. 

\medskip

\begin{lem}\label{lem:ws-cont-ext}
Let $n\in \BN$, $X$ be an operator space and $Y$ be a quotient operator space of $X$ with $q:X\to Y$ being the %v2
quotient map. 
For any weak-$^*$-continuous completely bounded map $\varphi: Y^* \to M_n(\BC)$ and $\epsilon > 0$, there exists a weak-$^*$-continuous map $\bar \varphi_\epsilon: X^*\to M_n(\BC)$ satisfying $\|\bar \varphi_\epsilon\|_\cb \leq \|\varphi\|_\cb +\epsilon$ and $\varphi = \bar \varphi_\epsilon\circ q^*$.  
\end{lem}

\medskip

In fact, it follows from the definition that $q^{(n)}$ will send the open unit ball of $M_n(X)$ onto the open unit ball of $M_n(Y)$. 
Thus, the conclusion follows from the fact that $M_n(Y)$ can be identified with the Banach space of weak-$^*$-continuous completely bounded maps from $Y^*$ to $M_n(\BC)$ and so is $M_n(X)$. 

\medskip

The following is another known fact that we needed. 
However, since we do not find the explicit statement in the literature, we give a proof here. 

\medskip

\begin{lem}\label{lem:norm-pos-cont}
	If $S$ is an operator system and $k\in \BN$, then for every $z\in M_k(S^{**})$, one has 
	$$\|z\|_{M_k(S^{**})} = \sup \big\{ \big\|\theta_{z}^{(n)}(\psi)\big\|: \psi\in M_n(S^*)^+\cap B_{M_n(S^*)}; n\in \BN \big\}.$$
\end{lem}
\begin{proof}
	It is clear that $\|z\|_{M_k(S^{**})}$ is larger than the quantity in the right hand side. 
	For the opposite inequality, we consider a completely order monomorphic complete isometry $\Psi: S \to \CL(\KK)$, and a normal faithful $^*$-representation $\Phi:\CL(\KK)^{**}\to \CL(\KH)$. 
	Clearly, $\Phi\circ \Psi^{**}:S^{**}\to \CL(\KH)$ is a weak-$^*$-continuous completely positive complete isometry. 
	For any $y\in M_k(\CL(\KH)) = \CL(\KH)\otimes M_k(\BC)$, we have 
	\begin{equation}\label{eqt:norm-L(H)}
		\|y\| = \sup \{\|(p\otimes I_k)y(p\otimes I_k)\|: p\in \CL(\KH) \text{ is a finite rank projection}\}.
	\end{equation}
	Thus, by identifying a $n$-dimensional subspace of $\KH$ with $\BC^{n}$ (for each $m\in \BN$), one has
	\begin{equation*}
		\|y\| = \sup\big\{ \|\phi^{(k)}(y)\|: \phi\in M_n(\CL(\KH)_*)^+\cap B_{M_n(\CL(\KH)_*)}; n\in \BN\big\}. 
	\end{equation*}
	Since the map from $\CL(\KH)_*$ to $S^*$ induced by the weak-$^*$-continuous map $\Phi\circ \Psi^{**}$ is a completely positive complete contraction, and $\Phi\circ \Psi^{**}$ is a complete isometry, the quantity in the right hand side of the displayed formula in the statement is larger than $\|z\|_{M_k(S^{**})}$. 
\end{proof}

\medskip

The last possible known fact that we need is the following. 

\medskip

\begin{lem}\label{lem:bidual}
	Let $X$ be a MOS. 
	The canonical complete isometry $\kappa_X:X\to X^{**}$ is completely order monomorphic. 
\end{lem}
\begin{proof}
	Obviously, $\kappa_X$ is completely positive. 
	Suppose that $x\in M_\infty(X)^\sa$ such that $\kappa_X^{(\infty)}(x)\in M_\infty(X^{**})^+$. 
	Then $\theta_{\kappa_X^{(\infty)}(x)}\in \CPB(X^*;M_\infty(\BC))$. 
	For every $m\in \BN$ and every $\varphi\in \Mor(X;M_m(\BC))$, there exists $f\in M_m(X^*)^+$ such that $\varphi = \theta_f$, and hence 
	$$\varphi^{(\infty)}(x) = \theta_x^{(m)}(f) = \theta_{\kappa_X^{(\infty)}(x)}^{(m)}(f)\geq 0.$$ 
	It now follows from \cite[Lemma 2.4(c)]{Ng-MOS} that $x\in M_\infty(X)^+$. 
\end{proof}

\medskip

\section{dual spaces of unital operator systems as dual operator systems}

\medskip

In the following, we present the main results of this article. 
As said in the above, the frameworks for our study are the notion of SMOS (as was recalled in Definition \ref{defn:SMOS}) and its dual version. 
Let us first introduce the latter in the following.

\medskip

\begin{defn}\label{defn:dual-quasi-op-sys}
(a) A SMOS (respectively, MOS) $D$ is called a \emph{dual SMOS} (respectively, \emph{dual MOS}) if 
\begin{itemize}
	\item there is a $^*$-invariant norm-closed subspace $D_\# \subseteq D^*$ (which will be called a \emph{predual of $D$}) with the canonical map $\upsilon_D: D \to (D_\#)^*$ being a surjective complete isometry;
	
	\item the matrix cone $M_\infty(D)^+$ is $\sigma\big(M_\infty(D),M_\infty(D_\#)\big)$-closed. 
\end{itemize}

\smnoind
(b) A dual MOS is called a \emph{dual quasi-operator system} (respectively, \emph{dual operator system}) if there is a weak-$^*$-homeomorphic (respectively, weak-$^*$-homeomorphic completely isometric) MOS isomorphism  $\Gamma$ from $D$ onto a weak-$^*$-closed subspace of $\CL(\KH)$, for a Hilbert space $\KH$. 
\end{defn}

\medskip

We will see in Proposition  \ref{prop:reg}(b) that it suffices to have the map $\Gamma$ in Definition \ref{defn:dual-quasi-op-sys}(b) being a weak-$^*$-continuous completely positive complete embedding (respectively, complete isometry). 
Moreover, by considering a faithful normal $^*$-representation of a von Neumann algebra, we may replace $\CL(\KH)$  in Definition \ref{defn:dual-quasi-op-sys}(b) by an arbitrary von Neumann algebra. 

%\medskip

%There is no reason to believe that a dual MOS $D$ which is at the same time a quasi-operator system is a dual quasi-operator system.  

\medskip

\begin{rem}\label{rem:dual-MOS}
(a) The assumption of $M_\infty(D)^+$ being $\sigma\big(M_\infty(D),M_\infty(D_\#)\big)$-closed is equivalent to $M_n(D)^+$ being $\sigma\big(M_n(D),M_n(D_\#)\big)$-closed for every $n\in \BN$. 

\smnoind
(b) Suppose that $D$ is a dual MOS with a predual $D_\#$. 
Since $D_\#$ is $^*$-invariant, the subspace $M_\infty(D_\#)$ of $M_\infty(D^*)$ is also $^*$-invariant. 
	Therefore, the involution $^*: M_\infty(D)\to M_\infty(D)$ is continuous under the weak-$^*$-topology $\sigma\big(M_\infty(D), M_\infty(D_\#)\big)$. 

\smnoind
(c) By \cite[Theorem 1.1]{BM}, every ``dual operator system'' in the sense of \cite{BM} is a dual operator system in our sense. 
Conversely, if $V$ is a dual operator system in our sense and it admits an Archimedean matrix order unit (in the sense of \cite{CE}), then it is a ``dual operator system'' in the sense of \cite{BM}. 

\smnoind
(d) As shown in  \cite[Proposition 2.3]{BM}, there could exist many different preduals of a ``dual operator system'' (even in the unital case). 
\end{rem}

\medskip

If $X$ is a \textcolor{magenta}{complete} SMOS, then $X^*$ is a dual SMOS. 
%v2
We will see in Corollary \ref{cor:reform-CP} below that every dual MOS is of the form $X^*$. 
However, even when $X$ is an operator system, $X^*$ may not be a MOS (see part (a) below). 
It is the reason why we need to consider dual SMOS, instead of dual MOS alone. 

\medskip

\begin{eg}\label{eg:SMOS}
(a) Let $X$ be the subspace of $M_2(\BC)$ consisting of matrices with their diagonal entries being zero.
Then $M_\infty(X)^+ = \{0\}$. 
Hence, $M_\infty(X^*)^+ = M_\infty(X^*)^\sa$, and $X^*$ is not a MOS. 
Therefore, there is no MOS isomorphism from $X^*$ to any operator system.  

\smnoind
(b) Let $Y$ be the subspace of the space $X$ in part (a), consisting of matrices whose two off-diagonal entries coincide. 
The canonical map from $\BC$ to $Y$ is an involution preserving surjective complete isometry. 
In other words, when we equip $\BC$ with the matrix cone $\{0\}$ without changing the operator space structure nor the involution, the resulting space is again a dual operator system.  

\smnoind
(c) Let $A$ be the unital $C^*$-algebra of all convergent sequences. 
We will see in Theorem \ref{thm:dual-oper-sys}(a) below that there exists a MOS isomorphism from $A^*$ to an operator system.
However, there is no MOS isomorphism from $A^*$ onto a \emph{unital} operator system. 
In fact, one may identify $A^*$ with $\ell_1(\BN\cup \{0\})$, by associating $(t_m)_{m\in \BN\cup \{0\}}\in \ell_1(\BN\cup \{0\})$ with the following functional on $A$:
$$(c_k)_{k\in \BN}\mapsto {\sum}_{k=1}^\infty t_k(c_k - {\lim}_n c_n)  + t_0{\lim}_n c_n.$$ 
Under this identification, one has $(A^*)^+ =\big\{(t_m)_{m\in \BN\cup \{0\}}\in \ell_1(\BN\cup \{0\})^+: t_0 \geq {\sum}_{n=1}^\infty t_n \big\}$. 
Fix an arbitrary element $(r_m)_{m\in \BN\cup \{0\}}\in (A^*)^+$.
We set $s_0 := r_0$ and $s_m:= r_m/m^2$.
Then $(s_m)_{m\in \BN\cup \{0\}}\in (A^*)^+$ but $(s_m)_{m\in \BN\cup \{0\}}\nleq \alpha \cdot (r_m)_{m\in \BN\cup \{0\}}$, for all $\alpha\in \RP$. 
This means that $(r_m)_{m\in \BN\cup \{0\}}$ cannot be an order unit for the ordered space $A^*$.
\end{eg}

\medskip

It is natural to ask when the dual space of a SMOS is a MOS (and hence is a dual MOS). 
Our first proposition gives an answer to this question. 
For this, we need the following lemma, which will also be needed in Proposition \ref{prop:reg} below.  
This lemma is basically the same as \cite[Lemma 2.4(a)]{Ng-MOS}, which follows more or less from the proof of \cite[Theorem 1.1]{BM}, together with the observation that any map $\varphi:X\to M_n$ satisfying Relation \eqref{rel-phi-T} is completely positive (one can verify this observation either through Relation \eqref{eqt:f-varphi-expand} below, or through the fact that any $n$-positive linear map from a MOS $X$ to $M_n(\BC)$ is completely positive). 

\medskip

\begin{lem}
	\label{lem-BM1.1}
	Let $X$ be a MOS. 
	Consider $n\in \mathbb{N}$ and $f\in (M_n(X)^*)^+$ with $\|f\| = 1$. 
	
	\smnoind
	(a) There exist $k\leq n$, a cyclic representation $(\mathbb{C}^{kn},\pi, \xi)$ of $M_n$ and $\varphi\in \Morc(X;M_k(\BC))$ satisfying 
	\begin{equation}\label{rel-phi-T}
	f(a^*vb)\ =\ \big\la \varphi^{(n)}(v)\pi(b)\xi, \pi(a)\xi\big\ra \qquad (v\in M_{n}(X); a,b\in M_{n}(\BC)). 
	\end{equation}
	
	\smnoind
	(b) If, in addition, $X$ is a dual MOS with a predual $X_\#$ and $f$ is $\sigma(M_n(X), M_n(X_\#))$-continuous, then the map $\varphi$ in part (a) is $\sigma(X, X_\#)$-continuous. 
\end{lem}
\begin{proof}
	(a) This follows from the second statement of \cite[Lemma 2.4(a)]{Ng-MOS}. 
	
	\smnoind
	(b) As $\xi$ is a cyclic vector for $\pi$, Relation \eqref{rel-phi-T} implies that $\varphi^{(n)}$ is $\sigma(M_n(X), M_n(X_\#))$-continuous (note that $M_n(\BC)\cdot M_n(X_\#)\cdot M_n(\BC)\subseteq M_n(X_\#)$). 
	Hence, $\varphi$ is $\sigma(X, X_\#)$-continuous. 
\end{proof}

\medskip

From Relation \eqref{rel-phi-T}, one obtains
\begin{equation}\label{eqt:f-varphi-expand}
	f(a^*vb)\ =\ \big\la \varphi^{(mn)}(v)(\pi\oplus \cdots \oplus \pi)(b)\xi^{(m)}, (\pi\oplus \cdots \oplus \pi)(a)\xi^{(m)}\big\ra, 
\end{equation}
for every $m\in \mathbb{N}$, $v\in M_{mn}(X)= M_m(M_n(X))$ and  $a,b\in M_{mn,n}(\BC) = M_n(\BC)\oplus \cdots \oplus M_n(\BC)$; 
here $\xi^{(m)} := (\xi, \dots, \xi)\in \BC^{kn}\oplus \cdots \oplus \BC^{kn}$.
In fact, if $a =(a_1,\dots,a_m)$, $b = (b_1,\dots b_m)$ and $v = [v_{i,j}]_{i,j=1,\dots,m}$, then one has 
$$f(a^*vb) = {\sum}_{i,j=1}^m f(a_i^*v_{i,j}b_j) = {\sum}_{i,j=1}^m  \big\la \varphi^{(n)}(v_{i,j})\pi(b_j)\xi, \pi(a_i)\xi\big\ra.$$

\medskip

\begin{prop}\label{prop:dual-sp-MOS}
Let $X$ be a SMOS. 
Then $X^*$ is \textcolor{magenta}{a MOS} if and only if the linear span of $M_\infty(X)^+$ is norm-dense in $M_\infty(X)$. 
\end{prop}
\begin{proof}
Suppose that the linear span of $M_\infty(X)^+$ is norm-dense in $M_\infty(X)$. 
	Consider $n \in \BN$ and $g\in M_n(X^*)^+ \cap - M_n(X^*)^+$. 
	By the definition, $\theta_g \in \CPB(X;M_n(\BC))\cap -\CPB(X;M_n(\BC))$. 
	This means that $\theta_g^{(\infty)}(M_\infty(X)^+) = \{0\}$, and the norm-density assumption tells us that  $\theta_g^{(\infty)}$ is zero; i.e., $g=0$. 
	
	Conversely, assume that the linear span, $E$, of $M_\infty(X)^+$ is not norm-dense in $M_\infty(X)$.
	This implies that $E\cap M_n(X)^\sa$ is not norm-dense in $M_n(X)^\sa$ for some $n\in \BN$. 
	Hence, there exists a non-zero bounded linear functional $f$ on $M_n(X)^\sa$ such that 
	$f\big(E\cap M_n(X)^\sa\big) = \{0\}$. 
	One may extend $f$ to a bounded complex linear functional $f\in M_n(X)^*\setminus \{0\}$ satisfying 
	\begin{equation}\label{eqt:F-vanish}
		f\big(E\cap M_n(X)\big) = \{0\}. 
	\end{equation}
	We assume that $\|f\| =1$.  
	The relation $M_n(X)^+\subseteq E\cap M_n(X)^\sa$ implies that $f\in M_n(X^*)^+$. 
	Let $\varphi\in \Morc(X; M_k(\BC))$ be the map as in Lemma \ref{lem-BM1.1}(a) for this positive linear functional $f$. 
	
	Since $a^*M_\infty(X)^+a\subseteq M_\infty(X)^+$ ($a\in M_\infty(\BC)$), a polarization argument tells us that 
	\begin{equation}\label{eqt:polar}
	a^*(E\cap M_{mn}(X))b\subseteq E\cap M_n(X) \qquad (a,b\in M_{mn,n}(\BC), m\in \BN).
	\end{equation}
	It then follows from 
%	\begin{equation}\label{eqt:Mn+-in-E}
	$M_{mn}(X)^+\subseteq E\cap M_{mn}(X)$
%	\end{equation}
	as well as  Relations \eqref{rel-phi-T}, \eqref{eqt:F-vanish} and \eqref{eqt:polar} that $-\varphi:X\to M_n(\BC)$ is $n$-positive (i.e. $-\varphi^{(n)}$ is positive), and hence is completely positive (alternatively, one can replace Relation \eqref{rel-phi-T} with  Relation \eqref{eqt:f-varphi-expand} to show directly that $\varphi^{(mn)}(v) = 0$ for every $v\in M_{mn}(X)^+$).
	Now, we consider $g\in B_{M_n(X^*)^\sa}\setminus \{0\}$ with $\varphi = \theta_g$. 
	As both $\varphi$ and $-\varphi$ are completely positive, we know that $g$ is a non-zero element in $M_n(X^*)^+ \cap - M_n(X^*)^+$.
\end{proof}

\medskip

In order to study a dual MOS $D$,  we need to consider certain nice ``representation'' of it. 
This ``representation'' mimics the corresponding construction in \cite{Ng-MOS} (whose idea comes from \cite{HN}). 
More precisely, for any $n\in \BN$, we denote by $\WQ^{D}_n$ the set of all $\sigma(D, D_\#)$-continuous maps in $\Morc(D;M_n(\BC))$. 
Consider the von Neumann algebra
\begin{equation}\label{eqt:def-N(D)}
N(D):={\bigoplus}_{n\in \BN}^{\ell_\infty} \ell_\infty\big(\WQ^{D}_n; M_n(\BC)\big), 
\end{equation}
%v2
where $\ell_\infty\big(\WQ^{D}_n; M_n(\BC)\big)$ is the algebra of bounded functions from the set $\WQ^{D}_n$ to the algebra $M_n(\BC)$. 

\medskip

Note that the predual of $\ell_\infty\big(\WQ^{D}_n; M_n(\BC)\big)$ is $\ell_1\big(\WQ^{D}_m; M_m(\BC)^*\big)$, where $\ell_1\big(\WQ^{D}_m; M_m(\BC)^*\big)$ is the set of functions $f:\WQ^{D}_m \to M_m(\BC)^*$ with 
\begin{equation}\label{eqt:def-l1}
{\sum}_{\varphi\in \WQ^{D}_m} \|f(\varphi)\|_{M_m(\BC)^*} < \infty.
\end{equation}
The duality is given by 
\begin{equation}\label{eqt:dual-l1}
g(f):= {\sum}_{\varphi\in \WQ^{D}_m} f(\varphi)(g(\varphi)) \qquad \big(g\in \ell_\infty\big(\WQ^{D}_n; M_n(\BC)\big); f\in \ell_1\big(\WQ^{D}_m; M_m(\BC)^*\big)\big).
\end{equation}
Thus, one has 
$N(D)_* = {\bigoplus}_{m\in \BN}^{\ell_1} \ell_1(\WQ^{D}_m; M_m(\BC)^*).$
For simplicity, we will identify $\ell_1(\WQ^{D}_m; M_m(\BC)^*)$ with a subspace of $N(D)_*$ in the canonical way.
%Moreover, we put 
%\begin{equation}\label{eqt:def:hat-D}
%\widehat{D} := \overline{\mu_D(D)}^{\sigma(N(D),N(D)_*) }. 
%\end{equation}

\medskip

\begin{prop}\label{prop:reg}
Let $D$ be a dual MOS with a predual $D_\#$. 
Define a map $\mu_{D}: D \to N(D)$ by 
$$\mu_D(x) := \big(\mu_D(x)_n\big)_{n\in \BN}\qquad (x\in D),$$ 
where the element $\mu_D(x)_n\in \ell_\infty\big(\WQ^D_n, M_n(\BC)\big)$ is given by 
%\begin{equation*}\label{eqt:def-jX}
$\mu_{D}(x)_n(\varphi) := \varphi(x)$ ($\varphi\in \WQ^{D}_n$).
%\end{equation*}
We equip the subspace $D^\reg:= \overline{\mu_D(D)}^{\sigma(N(D),N(D)_*) }$ with the dual operator system structure induced from $N(D)$.  

\smnoind
(a)	$\mu_D$ is a weak-$^*$-continuous completely order monomorphic complete contraction. 

\smnoind
(b) The following statements are equivalent. 
\begin{enumerate}
	\item $\mu_D$ is a complete embedding (respectively, complete isometry). 
	
	\item $D$ is a dual quasi-operator system (respectively, dual operator system).
	
	\item There exists a weak-$^*$-continuous completely positive  complete embedding (respectively, complete isometry) $\Psi$ from $D$ to some $\CL(\KH)$. 
\end{enumerate}

\smnoind
(c) If $D$ is a dual quasi-operator system, then $D^\reg = \mu_D(D)$, and $\mu_D$ is a weak-$^*$-homeomorphic MOS isomorphism from $D$ onto $\mu_D(D)$. 

\smnoind
(d) For any dual operator system $V$ and weak-$^*$-continuous map $\Phi\in \Morc(D;V)$, there is a unique weak-$^*$-continuous map $\overline{\Phi}\in \Morc\big(D^\reg;V\big)$ such that $\Phi = \overline{ \Phi}\circ \mu_D$.   
\end{prop}
\begin{proof}
	(a) Since elements in $\WQ^D_n$ are completely contractive for each $n\in \BN$, the map $\mu_D$ is a complete contraction. 
	Consider $n\in \BN$ and 
	$$f\in \ell_1\big(\WQ^D_n; M_n(\BC)^*\big)\subseteq N(D)_*.$$ 
	For any $\varphi\in \WQ^D_n$, one knows that $\varphi^*$  is a contraction from $M_n(\BC)^*$ to $D_\#$.
	Hence, Relation \eqref{eqt:def-l1} gives
	$${\sum}_{\varphi\in \WQ^D_n} \|\varphi^*(f(\varphi))\| \leq {\sum}_{\varphi\in \WQ^D_n} \|f(\varphi)\| < \infty,$$ 
	and thus, ${\sum}_{\varphi\in \WQ^D_n} \varphi^*(f(\varphi))$ exists in $D_\#$ (note that $D_\#$ is a Banach space by its definition).
	As
	$$\mu_D^*(f)(x) = f(\mu_D(x)) = {\sum}_{\varphi\in \WQ^D_n} f(\varphi)\big(\varphi(x)\big) \qquad (x\in D),$$
	we have $\mu_D^*(f) = {\sum}_{\varphi\in \WQ^D_n} \varphi^*(f(\varphi))\in D_\#$. 
	Thus, the norm-continuity of $\mu_D^*$ implies that $\mu_D^*(N(D)_*)\subseteq D_\#$.
	This produces the weak-$^*$-continuity of $\mu_D$. 
	
	Next, we will verify that $\mu_D$ is injective.
	In fact, as $\mu_D$ preserves the involution, it suffices to establish the injectivity of $\mu_D|_{D^\sa}$. 
	Consider $x\in D^\sa$ with $\mu_D(x) = 0$. 
	We want to show that $x \in D^+$. 
%	Then $\mu_D(x)\in N(D)^+\cap -N(D)^+$. 
	Suppose on the contrary that $x\notin D^+$. 
	Since $D^+$ is weak-$^*$-closed, we can find a weak-$^*$-continuous real linear functional $f: D^\sa \to \BR$ such that $f(D^+) \subseteq \RP$ but $f(x) < 0$. 
	As $D^\sa$ is weak-$^*$-closed (see Remark \ref{rem:dual-MOS}(b)), a rescaling of the complexification of $f$ belongs to $\WQ^D_1$. 
	However, this contradicts with the fact that $\mu_D(x)\in N(D)^+$. 
	Similarly, we have $x\in -D^+$. 
	Consequently, $x \in D^+\cap -D^+ = \{0\}$.

	Finally, we will establish that $\mu_D$ is completely order monomorphic. 
	Indeed, it is clear that $\mu_D$ is completely positive, because elements in $\WQ^D_n$ are completely positive. 
	Suppose on the contrary that 
	$$M_\infty(D)^+ \subsetneq \big(\mu_D^{(\infty)}\big)^{-1}\big(M_\infty(N(D))^+\big).$$ 
	As $\mu_D$ is injective and preserves the involuation, the above means that one can find an integer $n\in \BN$ as well as $v_0\in M_n(D)^\sa\setminus M_n(D)^+$ satisfying $\mu_{D}^{(n)}(v_0) \in M_n(N(D))^+$. 
	Moreover, since both $M_n(D)^\sa$ and $M_n(D)^+$ are weak-$^*$-closed, a similar argument as that for the injectivity of $\mu_D$ above gives $f\in \WQ^{M_n(D)}_1$ such that $f (v_0) \notin \mathbb{R}^+$. 
	Lemma \ref{lem-BM1.1}(b) will then produce an integer $k\leq n$, an element $\varphi\in \WQ^{D}_k$ and a cyclic representation $(\mathbb{C}^{kn}, \pi, \xi)$ for $M_n(\BC)$ such that 
	$$f(v)\ =\ \la \varphi^{(n)}(v)\xi, \xi\ra \qquad (v\in M_{n}(D)).$$
	Since $f(v_0) \notin \BR^+$, one has $\varphi^{(n)}(v_0)\notin M_{nk}(\BC)^+$. 
	Hence, we arrise at the contradiction that  $\mu_{D}^{(n)}(v_0)\notin M_n(N(D))^+$ (one may regard $\WQ^D_k\subseteq \WQ^D_n$).

	\smnoind
	(b) $(1)\Rightarrow (2)$ If $\mu_D$ is a complete embedding (respectively, complete isometry), then part (a) above and  Lemma \ref{lem:image-weak-st-cont-bdd-below} implies that $D$ is a dual quasi-operator system (respectively, dual operator system). 
	
	\noindent
	$(2)\Rightarrow (3)$
	This follows from the definition. 
	
	\noindent
	$(3)\Rightarrow (1)$
	By rescaling if necessary, one may assume that $\Psi\in \Morc(D; \CL(\KH))$ with 
	$$\gamma:= \sup \{\|\Psi^{(\infty)}(x)\|:x\in M_\infty(D); \|x\| = 1 \} > 0$$ (respectively, $\gamma = 1$). 
	For any $m\in \BN$ and $y\in M_m(\CL(\KH))$, we know from Lemma \ref{lem:norm-pos-cont} that 
$\|y\| = {\sup}_{\psi\in \WQ^{\CL(\KH)}_n; n\in \BN} \|\psi^{(m)}(y)\|.$
	Note that $\psi\circ \Psi \in \WQ^D_n$, for every $\psi\in \WQ^{\CL(\KH)}_n$.  
	Thus, when $x\in M_m(D)$, 
	$$\|\mu_D^{(m)}(x)\| \geq  {\sup}_{\psi\in \WQ^{\CL(\KH)}_n;n\in \BN} \big\|\psi^{(m)}\big(\Psi^{(m)}(x)\big)\big\| = \big\|\Psi^{(m)}(x)\big\|\geq \gamma \|x\|.$$
	
	\smnoind
	(c) This follows directly from parts (a) and (b) above, as well as Lemma \ref{lem:image-weak-st-cont-bdd-below}. 
	
	\smnoind
	(d) Since the uniqueness follows from the weak-$^*$-continuity of $\overline{\Phi}$, we will only establish its existence. 
	Moreover, by replacing $V$ with $\overline{\Phi(D)}^{\sigma(V,V_\#)}$ if necessary, we may assume that $\Phi(D)$ is weak-$^*$-dense in $V$. 
	Hence, if $\psi_1, \psi_2\in \WQ^V_n$ satisfying $\psi_1\circ \Phi = \psi_2\circ \Phi$, then $\psi_1 = \psi_2$. 
	This means that, under the weak-$^*$-density assumption of $\Phi(D)$, one may regard $\WQ^V_n$ as a subset of $\WQ^D_n$, via compositions with $\Phi$. 
	We obtain from this a complete isometry $\Psi_n: \ell_1\big(\WQ^V_n;M_n(\BC)^*\big) \to \ell_1\big(\WQ^D_n;M_n(\BC)^*\big)$ such that for $f\in \ell_1\big(\WQ^V_n;M_n(\BC)^*\big)$, 
	$$\Psi_n(f)(\phi) := \begin{cases}
	f(\psi) & \text{when }\phi = \psi \circ \Phi \text { for an element }\psi\in \WQ^V_n\\
	0 & \text{otherwise}
	\end{cases} \qquad (\phi \in \WQ^D_n).$$ 
	The above then produces a completely positive complete isometry $\Psi: N(V)_*\to N(D)_*$ %it is the dual map of a $^*$-homomorphism
	satisfying
	$$\Psi\big((f_n)_{n\in \BN}\big) = \big(\Psi_n(f_n)\big)_{n\in \BN} \qquad ((f_n)_{n\in \BN}\in N(V)_*).$$ 
	Pick any $x\in D$. 
	As $\Psi^*(\mu_D(x)) \in N(V)$, one has $\Psi^*(\mu_D(x)) = \big(\Psi^*\big(\mu_D(x)\big)_n\big)_{n\in \BN}$, where $\Psi^*\big(\mu_D(x)\big)_n\in \ell_\infty\big(\WQ^V_n; M_n(\BC)\big)$ is given by
	$$\Psi^*\big(\mu_D(x)\big)_n= \mu_D(x)_n \circ \Psi_n \in \ell_1\big(\WQ^V_n,M_n(\BC)^*\big)^*;$$
	here, we identify $\ell_\infty\big(\WQ^D_n,M_n(\BC)\big)$ with $\ell_1\big(\WQ^D_n,M_n(\BC)^*\big)^*$ as in Relation \eqref{eqt:dual-l1}. 
	If $\psi\in \WQ^V_n$ and $\alpha\in M_n(\BC)^*$, we define  $\delta_{\psi}^\alpha\in \ell_1(\WQ^V_n,M_n(\BC)^*)$ by $\delta_{\psi}^\alpha(\psi) := \alpha$ and $\delta_{\psi}^\alpha(\phi) := 0$ when $\phi\neq \psi$. 
	Then 
	$$\big(\mu_D(x)_n \circ \Psi_n\big)(\delta_{\psi}^\alpha) = \mu_D(x)_n(\delta_{\psi\circ \Phi}^\alpha) = \alpha\big(\psi(\Phi(x))\big) = \mu_V\big(\Phi(x)\big)_n(\delta_{\psi}^\alpha)$$ 
	(the second and the third equalities follow from Relation \eqref{eqt:dual-l1}). 
	These show that $\Psi^*\circ \mu_D = \mu_V\circ \Phi$. 
	
	Now, parts (a), (b) and (c) tells us that $\mu_V(V)$ is weak-$^*$-closed in $N(V)$ and $\mu_V:V\to \mu_V(V)$ is a weak-$^*$-homeomorphic completely order monomorphic complete isometry. 
	Since $\Psi^*$ is weak-$^*$-continuous, we know that $\Psi^*\big(D^\reg\big)\subseteq \mu_V(V)$. 
	Moreover, as $\Psi^*$ is a completely positive complete contraction, the map $\overline{\Phi}:= \mu_V^{-1}\circ \Psi^*|_{D^\reg}$ will satisfy the requirement.  
\end{proof}

\medskip

By part (a) above, there always exists a new matrix norm on $D$ such that $D$ becomes a subsystem of a dual operator system without changing the matrix cone on $D$. 
Moreover, part (b) tells us that this new matrix norm is equivalent to the original one on $D$ if and only if $D$ is a dual quasi-operator system. 
On the other hand, parts (a) and (d) above tell us that $D^\reg$ can be regarded as a universal ``dual operator system cover'' of $D$.

\medskip

The two parts of following corollary are direct consequences of Proposition  \ref{prop:reg}(b).

\medskip

\begin{cor}\label{cor:reg}
(a) Let $T$ be a dual quasi-operator system (respectively, dual operator system). 
If $S$ is the involutive dual operator space $T$ equipped with a weak-$^*$-closed matrix cone smaller than the one on $T$, then $S$ is again a dual quasi-operator system (respectively, dual operator system). 

\smnoind
(b) Suppose that $X$ is a SMOS with  $X^*$ being a dual quasi-operator system. 
If $Y$ is the operator space $X$ equipped with a closed matrix cone larger than the one on $X$, then $Y^*$ is a dual quasi-operator system. 
\end{cor}

\medskip

For a dual MOS $D$ with a predual $D_\#\subseteq D^*$, we will equip $D_\#$ with the SMOS structure induced from $D^*$. 
%By the definition and \cite[p.41]{ER}, we know that 
%\begin{equation}\label{eqt:reform-CB}
%\{\varphi\in \CB(D_\#; M_\infty(\BC)): \|\varphi\|_\cb \leq 1 \} = \{\theta_s: s\in B_{M_\infty(D)}\}. 
%\end{equation}
The following result tells us that the MOS structure on $D$ is the dual SMOS structure coming from $D_\#$. 

\medskip

\begin{cor}\label{cor:reform-CP}
Let $D$ be a dual MOS with a predual $D_\#$.
The canonical surjection $\upsilon_D: D \to (D_\#)^*$ is a completely order monomorphic complete isometry. %v2
\end{cor}
\begin{proof}
By the definition, $\upsilon_D$ is a complete isometry. 
It remains to verify that $\CPB(D_\#;M_\infty(\BC)) = M_\infty(D)^+$. 
Consider $n\in \BN$ as well as $s = [s_{i,j}]\in M_n(D)$. 
If  $s\in M_n(D)^+$, then for any $m\in \BN$ and $\omega = [\omega_{k,l}]\in M_m(D_\#)^+\subseteq M_m(D^*)^+$, one has 
$$\theta_s^{(m)}(\omega) = \theta_\omega^{(n)}(s)\in M_{mn}(\BC)^+, $$
which implies that $\theta_s\in \CPB(D_\#;M_n(\BC))$. 

Conversely, suppose that $\theta_s\in \CPB(D_\#;M_n(\BC))$. 
Take $m\in \BN$ and $\varphi\in \WQ^D_m \subseteq \CPB(D;M_m(\BC)) = \theta(M_m(D^*)^+)$. 
As $\varphi$ is weak-$^*$-continuous, Lemma \ref{lem:weak-st-cont} gives $\omega\in M_m(D_\#)$ with $\varphi = \theta_\omega$. 
Thus, $\omega\in M_m(D^*)^+ \cap M_m(D_\#) = M_m(D_\#)^+$. 
This means that $\varphi^{(n)}(s) = \theta_s^{(m)}(\omega) \in M_{mn}(\BC)^+$. 
Since $m$ and $\varphi$ are arbitrary, one has $\mu_D^{(n)}(s) \geq 0$, and it follows from Proposition  \ref{prop:reg}(a) that $s\in M_n(D)^+$. 
\end{proof}

\medskip

Suppose that $X$ is an operator space with its matrix norm being denoted by $\|\cdot\|$.
If $\rn$ is another matrix norm  on $X$, we denote by $\rn^*$ the dual matrix norm on $(X, \lambda)^*$. 
As usual, we say that $\rn$ is \emph{equivalent to} $\|\cdot\|$, and is denoted by $\rn\sim \|\cdot\|$,  if the identity map from $(X, \rn)$ to $(X, \|\cdot\|)$ is a complete embedding. 
In this case, we have $(X, \lambda)^* = X^*$. 
Let us also set 
$$\CN_X:= \{\rn^*: \rn \text{ is a matrix norm on } X \text{ with }\rn\geq \|\cdot\| \text{ and } \rn \sim \|\cdot\| \}.$$
Obviously, every matrix norm on $X$ that is equivalent to $\|\cdot\|$ can be rescaled to produce a dual matrix norm in $\CN_X$. 

\medskip

The following theorem gives a characterization of the situation when the dual SMOS $X^*$ of \textcolor{magenta}{a complete} operator system $X$ is a dual quasi-operator system. 
In this case, we also obtain a ``universal'' dual operator system that is MOS isomorphic to $X^*$. 

\medskip

\begin{thm}\label{thm:quasi-oper-sys}
Let $X$ be a \textcolor{magenta}{complete} SMOS. 
We denote $B_{M_\infty(X)}^+ := M_\infty(X)^+\cap B_{M_\infty(X)}$,  
$$U_X:=\big\{u - v: u,v\in B_{M_\infty(X)}^+ \big\}, \qquad X^\rd:=\mu_{X^*}(X^*),$$  
and
$\CN_X^\mathrm{sys}:= \{\rn^*\in \CN_X: X^* \text{ becomes a dual operator system under the dual matrix cone and }\rn^* \}.$

\smnoind
(a) The following statements are equivalent. 
\begin{enumerate}
	\item $X^*$ is a dual quasi-operator system.
	
	\item $U_X$ is a zero neighborhood of the normed space $M_\infty(X)^\sa$.
	
	\item The norm closure, $\overline{U_X}$, of $U_X$ is a zero neighborhood of $M_\infty(X)^\sa$.
	
	\item $\CN_X^\mathrm{sys} \neq \emptyset$. 
\end{enumerate}

\smnoind
(b) Suppose that $X^*$ is a dual quasi-operator system, and 
	$$\|f\|^\rd:= \|\mu_{X^*}^{(\infty)}(f)\| = \sup \big\{\big\|\theta_f^{(n)}(x)\big\|: x\in M_n(X)^+; \|x\|\leq 1; n\in \BN\big\} \quad (f\in M_\infty(X^*)).$$ 
Then $\|\cdot \|^\rd$ is the largest element in $\CN_X^\mathrm{sys}$.  
\end{thm}
\begin{proof}
(a) $(1)\Rightarrow (2)$. 
Let $\Gamma:X^* \to \CL(\KH)$ be a map such that $\Gamma:X^* \to \Gamma(X^*)$ a weak-$^*$-homeomorphic MOS isomorphism. 
Set 
$$Y := \{\omega|_{\Gamma(X^*)}: \omega\in \CL(\KH)_*\},$$ 
and equip it with the MOS structure induced from $Y\subseteq \Gamma(X^*)^*$. 
Then $Y^*$ can be identified with the MOS $\Gamma(X^*)$ (by Corollary \ref{cor:reform-CP}). 
Moreover, as $\Gamma$ is also a weak-$^*$-homeomorphism, it induces a MOS isomorphism from $Y$ onto $X$. 
Thus, $U_X$ contains a zero neighborhood of $M_\infty(X)^\sa$ if and only if $U_Y$ contains a zero neighborhood of $M_\infty(Y)^\sa$. 

Fix $n\in \BN$ and $y\in M_n(Y)^\sa$ with $\|y\| \leq 1$. 
Then $\theta_y$ is a weak-$^*$-continuous complete contraction from $Y^* = \Gamma(X^*)\subseteq \CL(\KH)$ to $M_n(\BC)$. 
It follows from Lemma \ref{lem:ws-cont-ext} that $\theta_y$ extends to a self-adjoint weak-$^*$-continuous map $\bar \theta_y: \CL(\KH)\to M_n(\BC)$ with $\|\bar \theta_y\|_\cb \leq 3$.
By Lemma \ref{lem:decomp}, there exist weak-$^*$-continuous completely positive map $\psi_1, \psi_2$ from $\CL(\KH)$ to $M_n(\BC)$ such that $\|\psi_1\|_\cb, \|\psi_2\|_\cb \leq 3$ and $\bar \theta_y = \psi_1 - \psi_2$. 
Now, for $i\in \{1,2\}$, we consider $y_i\in M_\infty(Y)^+$ to be the element satisfying $\theta_{y_i} = \psi_i|_{Y^*}$ (see Lemma \ref{lem:weak-st-cont}). 
In this case, we have  $\|y_1\|, \|y_2\| \leq 3$ and $y = y_1-y_2$. 
	
	\noindent
	$(2)\Rightarrow (3)$. 
	This is obvious. 
	
	\noindent
	$(3)\Rightarrow (4)$. 
	By the hypothesis, there is $\gamma > 0$ with 
	\begin{equation}\label{eqt:weak-bdd-decomp}
	B_{M_\infty(X)}\subseteq \overline{\big\{(u-v)+ \mathrm{i}(x-y): u,v,x,y \in \frac{1}{\gamma}B_{M_\infty(X)}^+\big\}}. 	
	\end{equation}
	Hence, Proposition \ref{prop:dual-sp-MOS} tells us that $X^*$ is a dual MOS, and we may apply Proposition  \ref{prop:reg}(a) to ensure that $\mu_{X^*}$ is a weak-$^*$-continuous completely order monomorphic complete contraction. 
	We claim that $\mu_{X^*}$ is a complete embedding. 
	
	In fact, consider arbitrary $m\in \mathbb{N}$ and $f\in M_m(X^*)$ with $\|f\| = 1$. 
	For every $\epsilon \in(0,1/2)$, one can find $n_0 \in \BN$ and $s\in M_{n_0}(X)$ with $\|s\|=1$ such that $\big\|\theta_f^{(n_0)}(s)\big\| >  1 - \epsilon$.  
	Relation \eqref{eqt:weak-bdd-decomp} gives $n_1\geq n_0$ and $u_1, u_2, u_3, u_4\in M_{n_1}(X)^+$ satisfying $\|u_i\| \leq 1/\gamma$ ($i=1,2,3,4$)  and $\|s - (u_1 - u_2) + \mathrm{i}(u_3-u_4)\| < \epsilon$. 
	Thus, $\big\|\theta_f^{(n_0)}\big(s - (u_1 - u_2) + \mathrm{i}(u_3-u_4)\big)\big\| < \epsilon$, and so, one can find $v_\epsilon\in B_{M_{n_1}(X)}^+$ with 
	\begin{equation}\label{eqt:f-v}
	\big\|\theta_f^{(n_1)}(v_\epsilon)\big\| > (1 - 2\epsilon)\gamma/4.
	\end{equation}
	Obviously,  $\theta_{v_\epsilon}\in \WQ^{X^*}_{n_1}$. 
	Moreover, as $\theta_{v_\epsilon}^{(m)}(f) = \theta_f^{(n_1)}(v_\epsilon)$, we know that 
	$$\sup \{\|\varphi^{(m)}(f)\|: \varphi\in \WQ^{X^*}_n; n\in \BN\} \geq \gamma/4.$$ 
	In other words, $\big\|\mu_{X^*}^{(m)}(f)\big\| \geq \gamma/4$. 
	As $m$ and $f$ are arbitrarily chosen, we know that $\mu_{X^*}$ is a complete embedding. 
	
	Now, Proposition  \ref{prop:reg}(c) implies that $X^\rd$ is a weak-$^*$-closed subspace of $N(X^*)$, and the MOS isomorphism  $\mu_{X^*}$ is a weak-$^*$-homeomorphism from $X^*$ onto $X^\rd$. 
	Let 
	$$X_0:= N(X^*)_*/\{\omega \in N(X^*)_*: \omega(z) = 0 \text{ for each } z\in X^\rd\}.$$ 
	Then $\mu_{X^*}$ produces a bijective completely bounded below complete contraction $\Phi: X_0\to X$. 
	Hence, if $\lambda$ is the matrix norm on $X$ induced from $\Phi$, then $\lambda^* \in \CN_X$. 
	Furthermore, as $\mu_{X^*}$ is a MOS isomorphism from $X^*$ to $X^\rd$, we know that $\lambda^*\in \CN_X^\mathrm{sys}$. 
	
	\noindent
	$(4)\Rightarrow (1)$. 
	Pick any $\lambda^*\in \CN_X^\mathrm{sys}$ and set $Z$ to be dual operator space $(X^*,\lambda^*)$ (equipped with the dual matrix cone). 
	It follows that there is a weak-$^*$-continuous completely positive complete isometry $\Gamma:Z\to \CL(\KH)$ for some Hilbert space $\KH$. 
	Thus, $\Gamma: X^*\to \CL(\KH)$ will be a weak-$^*$-continuous complete positive complete embedding, and $X^*$ is a dual quasi-operator system (by Proposition  \ref{prop:reg}(b)).

	\smnoind
	(b) It follows from parts (a), (b) and (c) of Proposition  \ref{prop:reg} that $\mu_{X^*}:X^* \to X^\rd$ is a weak-$^*$-homeomorphic MOS isomorphism. 
	Hence, the matrix norm $\|\cdot\|^\rd$, which coincides with the matrix norm on $X^*$ induced from $\mu_{X^*}$, belongs to $\CN_X^\mathrm{sys}$. 
	Suppose that $\rn^*\in \CN_X^\mathrm{sys}$. 
	Let $V$ be the dual operator system $(X^*, \rn^*)$, and $\Psi: (X,\rn)\to (X, \|\cdot\|)$ be the identity map. 
	Then $\Psi$ is both a complete contraction and a complete embedding (by the definition of $\CN_X$). 
	Thus, $\Psi^*: X^* \to V$ is a weak-$^*$-continuous completely order monomorphic complete contraction. 
	Proposition  \ref{prop:reg}(d) will then gives a completely positive complete contraction
	$\overline{\Psi^*}: X^\rd\to V$ with $\Psi^* = \overline{\Psi^*}\circ \mu_{X^*}$. 
	In particular,  $\|\cdot\|^\rd$ larger than $\lambda^*$.   
\end{proof}

\medskip

Theorem \ref{thm:quasi-oper-sys}(a) tells us that if $X$ is a \textcolor{magenta}{complete} SMOS, then $X^*$ is a dual quasi-operator system if and only if the ordered normed space $M_\infty(X)^\sa$ satisfies a stronger version of the ``bounded decomposition property'' in the sense of \cite{Bon} (\textcolor{magenta}{see \cite[Lemma 1]{Bon}}). 
However, since the term ``bounded decomposition property'' has many different meaning even within the subject of ordered normed space (see e.g. \cite[p.44]{AE}), we will not use this term here.

%v2 example added 

\medskip
\begin{eg}
Let $\KH$ be a Hilbert space. 
	It follows from \cite[Corollary 2.8]{Ng-MOS} that $\CL(\KH)^*$ is not an operator system. 
	However, $\CL(\KH)^*$ is a dual quasi-operator system (because $\CL(\KH)$ satisfies Condition (2) of Theorem \ref{thm:quasi-oper-sys}(a)). 
	Consider $t = [t_{i,j}]\in M_m(\CL(\KH)^*)$. 
	As in Theorem \ref{thm:quasi-oper-sys}(b), 
	$$\|t\|^\rd = \sup \big\{\big\|\big[t_{i,j}(x_{k,l})\big]\big\|_{M_{mn}(\BC)}: [x_{k,l}]\in M_n(\CL(\KH))^+\cap B_{M_n(\CL(\KH))}; n\in \BN\big\}.$$ 

	Suppose that $t\in M_m(\CL(\KH)^*)^+$.
	Then $\theta_t$ is a completely positive map from $\CL(\KH)$ to $M_m(\BC)$, and 
	$$\|t\|_{M_m(\CL(\KH)^*)} = \|\theta_t\| = \big\|\theta_t(1)\big\| = \big\|[t_{i,j}(1)]\big\|_{M_m(\BC)} \leq  \|t\|^\rd \leq \|t\|_{M_m(\CL(\KH)^*)},$$
which means that $\|t\|^\rd = \|t\|_{M_m(\CL(\KH)^*)}$. 

We may identify  $\CL(\KH)^\rd$ with a weak-$^*$-closed subspace of some $\CL(\KK)$. 
For any $s,t\in (\CL(\KH)^\rd)^+\setminus \{0\}$, the images of $s$ and $t$ in  $\CL(\KK)^+$ cannot have orthogonal ranges; otherwise, one has the contradiction that 
$$\|s+t\|^\rd = \max\{\|s\|^\rd,\|t\|^\rd \}\neq \|s\|^\rd + \|t\|^\rd = \|s\|_{\CL(\KH)^*} + \|t\|_{\CL(\KH)^*} = \|s+t\|_{\CL(\KH)^*}.$$ 
\end{eg}

%v2 Remark 3.11 moved to Appendix A.1 

\medskip

\begin{defn}\label{defn:dualizable}
	A quasi-operator system $T$ is said to be \emph{dualizable} if the dual SMOS $T^*$ is a dual quasi-operator system. 
\end{defn}

\medskip

\textcolor{magenta}{Note that when $T$ is dualizable, it is complete (because of the definition of the dual SMOS.)}

\medskip

When $T$ is dualizable, we may sometimes ignore the map $\mu_{T^*}$ and identify $T^\rd$ with the dual matrix ordered space $T^*$, equipped with a new dual matrix norm.

%\medskip

%Clearly, if $S$ is a quasi-operator system and there is a surjective completely order monomorphic complete embedding from $S$ onto a dualizable quasi-operator system $T$, then $S$ is also dualizable. 

\medskip 

The following are some duality results concerning dualizable quasi-operator systems.

\medskip

\begin{prop}\label{prop:predual}
(a) If $X$ is a quasi-operator system (respectively, operator system), then $X^{**}$ is a dual quasi-operator system (respectively, dual operator system). 

\smnoind
(b) If $T$ is a dualizable quasi-operator system, then $T^*$ is a dualizable dual quasi-operator system. 
\end{prop}
\begin{proof}
(a) Let $\Psi: X\to \CL(\KH)$ be a map with $\Psi: X\to \Psi(X)$ being a (respectively, completely isometric) MOS isomorphism. 
Then $\Psi^{**}: X^{**}\to \CL(\KH)^{**}$ is a weak-$^*$-continuous completely positive complete embedding (respectively, complete isometry).
By Lemma \ref{lem:image-weak-st-cont-bdd-below}, $T := \Psi^{**}( X^{**})$ is a weak-$^*$-closed subspace of $\CL(\KH)^{**}$, and $\Psi^{**}$ is a weak-$^*$-homeomorphism from $X^{**}$ onto $T$. 
Since the matrix cone $M_\infty(X^{**})^+$ is weak-$^*$-closed in $M_\infty(X^{**})$, we know that the subcone $(\Psi^{**})^{(\infty)}(M_\infty(X^{**})^+)$ of 
$$M_\infty(T)^+:= M_\infty(T)\cap M_\infty(\CL(\KH)^{**})^+$$ 
is again weak-$^*$-closed in $M_\infty(T)^\sa$. 
As $T$ is a dual operator system under the matrix cone $M_\infty(T)^+$, it
follows from Corollary \ref{cor:reg}(a) that $T$ is a dual operator system under the matrix cone $(\Psi^{**})^{(\infty)}(M_\infty(X^{**})^+)$. 
From this, we see that $X^{**}$ is a dual quasi-operator system (respectively, dual operator system). 

\smnoind
(b) This follows directly from part (a).
\end{proof}

\medskip

In the proof of part (a) above, it may not be true that $(\Psi^{**})^{(\infty)}(M_\infty(X^{**})^+) = M_\infty(T)^+$, in general.
In fact, even if $X\subseteq \CL(\KH)$, there is no guarantee that the matrix cone on $X^{**}$ induced from $\CL(\KH)^{**}$ is independent of the embedding $X\subseteq \CL(\KH)$. 
However, if $(X^1, \iota_X)$ is the unitization of $X$ (see Proposition \ref{prop:unital}) and $\Phi: X^1\to  \CL(\KH)$ is a unital complete embedding, then $\Phi^{**}\circ \iota_X^{**}: X^{**}\to \CL(\KH)^{**}$ is  completely order monomorphic.

\medskip

%v2 put two paragrahs into a remark. 
\begin{rem}
	Let $D$ be a dual SMOS with a predual $D_\#$. 
	By Proposition \ref{prop:predual}(a) and Theorem \ref{thm:quasi-oper-sys}(a), $U_D$ contains a zero neighborhood of $M_\infty(D)^\sa$ if and only if $D_\#$ is a quasi-operator system.
\end{rem}
\medskip

%v2 

\begin{cor}\label{cor:bidual-quasi-os}
Let $X$ be a  SMOS. 

\smnoind
(a) The following statements are equivalent. 

\begin{enumerate}
	\item $X$ is a quasi-operator system (respectively, an operator system). 
	
	\item $X^{**}$ is a dual quasi-operator system (respectively, dual operator system). 

	\item $X^{**}$ is a quasi-operator system (respectively, operator system). 
	
	\item $U_{X^*}$ is a zero neighborhood of $M_\infty(X^*)^\sa$.
\end{enumerate}

\smnoind
(b) If \textcolor{magenta}{$X$ is complete}, and  both $X$ and $X^*$ are quasi-operator systems, then $X^*$ is a dual quasi-operator system. 
%
%\smnoind
%(c) If $X$ is a quasi-operator system, then $X$ is dualizable if and only if $X^{**}$ is dualizable. 
\end{cor}
\begin{proof}
(a) $(1)\Rightarrow (2)$
This follows from Proposition \ref{prop:predual}(a). 

\smnoind
$(2)\Rightarrow (3)$ This is obvious. 

\smnoind
$(3)\Rightarrow (1)$
This follows from Lemma \ref{lem:bidual}. 

\smnoind
$(2)\Leftrightarrow (4)$ This follows from Theorem \ref{thm:quasi-oper-sys}(a). 

\smnoind
(b) Since both $X$ and $X^*$ are quasi-operator system, it follows from \cite[Corollary 5]{Han} (note that operator systems in \cite{Han} are quasi-operator systems in this paper) that there is a MOS isomorphism from $X$ onto a matrix regular operator space $Y$. 
Since by the definition of matrix regularity, $2U_Y$ contains the open unit ball of $M_\infty(Y)^\sa$, we see that $U_X$ is a zero neighborhood of $M_\infty(X)^\sa$ and Theorem \ref{thm:quasi-oper-sys}(a) tells us that $X^*$ is a dual quasi-operator system. 
%
%\smnoind
%(c) The forward implication follows from Proposition \ref{prop:predual}(b). 
%Conversely, suppose that $X^{**}$ is dualizable. 
%Then $X^{***}$ is a dual quasi-operator system and part (a) tells us that $X^*$ is a quasi-operator system. 
%Now, part (b) tells us that $X^*$ is a dual quasi-operator system; i.e., $X$ is dualizable. 
\end{proof}

\medskip

Note that in part (a) above, we do not assume that $X^*$ is a MOS.

%v2 a paragraph removed

\medskip

Let $S$ be a dualizable quasi-operator system. 
Then the bijection $\mu_{S^*}:S^* \to S^\rd$ is a MOS isomorphism. 
It follows that the weak-$^*$-homeomorphism $\mu_{S^*}^*: (S^\rd)^*\to S^{**}$ is also a MOS isomorphism and that the quasi-operator system $S^\rd$ is dualizable (see Proposition \ref{prop:predual}(b)). 
Hence, $\mu_{(S^\rd)^*}: (S^\rd)^* \to (S^\rd)^\rd$ is a weak-$^*$-homeomorphic MOS isomorphism. 
% v2
We call the map 
\begin{equation}\label{eqt:def-tau}
\tau_S:= \mu_{(S^\rd)^*}\circ (\mu_{S^*}^*)^{-1}: S^{**}\to (S^\rd)^\rd
\end{equation}
the \emph{canonical bijection}. 

\medskip

\begin{lem}\label{lem:quasi-oper-sys}
Let  $S$ be a dualizable quasi-operator system. 

\smnoind
(a) The canonical bijection $\tau_S:S^{**} \to (S^\rd)^\rd$ is a weak-$^*$-homeomorphic MOS isomorphism. 

\smnoind
(b) If $S$ is an operator system, then $\tau_S^{-1}: (S^\rd)^\rd\to S^{**}$ is a complete contraction.
\end{lem}
\begin{proof}
By the discussion preceding this lemma, we only need to verify part (b).
Suppose that $k\in \BN$ and $x\in M_k((S^\rd)^*)$. 
The definition of $\mu_{(S^\rd)^*}$ as well as  Lemma \ref{lem:weak-st-cont} imply that 
\begin{align}\label{eqt:bidual-norm}
	\big\|(\mu_{(S^\rd)^*})^{(k)}(x)\big\|_{M_k((S^\rd)^\rd)} 
	&= \sup \{\|\psi^{(k)}(x)\|: \psi \in \WQ^{(S^\rd)^*}_n; n\in \BN \}
	\nonumber\\
	& = \sup \{\|\theta_x^{(n)}(\varphi))\|: \varphi \in M_n(S^\rd)^+\cap B_{M_n(S^\rd)}; n\in \BN \}.
\end{align}
On the other hand, Lemma \ref{lem:norm-pos-cont} tells us that
\begin{align}\label{eqt:norm-bidual}
	\big\|(\mu_{S^*}^*)^{(k)}(x)\big\|_{M_k(S^{**})} 
	& = \sup \big\{ \big\|\theta_{(\mu_{S^*}^*)^{(k)}(x)}^{(n)}(\phi)\big\|: \phi\in M_n(S^*)^+\cap B_{M_n(S^*)}; n\in \BN \big\}
\end{align}
Since $\mu_{S^*}$ is a completely order monomorphic complete contraction (see Proposition  \ref{prop:reg}(a)), we know that $\mu_{S^*}^{(n)}\big(M_n(S^*)^+\cap B_{M_n(S^*)}\big) \subseteq  M_n(S^\rd)^+\cap B_{M_n(S^\rd)}$. 
From these, we obtain required inequality
$$\big\|(\mu_{S^*}^*)^{(k)}(x)\big\|_{M_k(S^{**})}\leq \big\|(\mu_{(S^\rd)^*})^{(k)}(x)\big\|_{M_k((S^\rd)^\rd)}.$$ 
\end{proof}

% add in v2
\medskip

If $D$ and $E$ are dual operator systems, we use $\Morw(D,E)$ (respectively, $\Ccpw(D,E)$) to denote the set of weak-$^*$-continuous completely positive completely bounded maps (respectively, complete contractions) from $D$ to $E$.

\medskip

\begin{thm}\label{thm:dualization-functor}
Let $S$ and $T$ be dualizable operator systems. 

\smnoind
(a) Suppose that $\phi\in \Mor(S,T)$. 
There is  a unique $\phi^\rd\in \Morw(T^\rd, S^\rd)$ satisfying $\phi^\rd\circ \mu_{T^*} = \mu_{S^*}\circ \phi^*$. 
If, in addition, $\phi\in \Morc(S,T)$, then $\phi^\rd\in \Morc_w(T^\rd, S^\rd)$. 

\smnoind
(b) The assignment $\phi\mapsto \phi^\rd$ is a bijection from  $\Mor(S,T)$ onto $\Morw(T^\rd, S^\rd)$. 

\smnoind
(c) If $\phi:S\to T$ is a MOS isomorphism, then $\phi^\rd$ is a weak-$^*$-homeomorphic MOS isomorphism. 
Conversely, any weak-$^*$-continuous MOS isomorphism $\chi: T^\rd\to S^\rd$ is of the form $\phi^\rd$ for a MOS isomorphism $\phi$. 

\smnoind
(d) For a dualizable dual operator system $D$, there is a dualizable operator system $S$ and a weak-$^*$-homeomorphic MOS isomorphism $\Phi$ from $S^\rd$ onto $D$.  

\smnoind
(e) Suppose that both  $\tau_S\circ \kappa_S$ and $\tau_T\circ \kappa_T$ are completely isometric, and $\chi\in \Morc_w(T^\rd;S^\rd)$. 
There exists $\phi\in \Morc(S;T)$ satisfying $\chi = \phi^\rd$. 
If, in addition,  $\chi$ is a completely isometric MOS isomorphism, then $\phi$ is a completely isometric MOS isomorphism. 
\end{thm}
\begin{proof}
(a) The uniqueness of $\phi^\rd$ follows from the fact that both $\mu_{S^*}:S^* \to S^\rd$ and $\mu_{T^*}:T^*\to T^\rd$ are bijective. 
For the existence, note that if $\phi = 0$, then we simply take $\phi^\rd := 0$. 
On the other hand, if $\phi\neq 0$, then the map $\mu_{S^*}\circ \phi^*/\|\phi^*\|\in \Morc_w(T^*;S^\rd)$ (see Proposition \ref{prop:reg}(a)), and the existence of $\phi^\rd$ comes from Proposition \ref{prop:reg}(d). 

Now, if $\phi$ is a complete contraction, then so is $\phi^*$, and we may use $\mu_{S^*}\circ \phi^*$ instead of $\mu_{S^*}\circ \phi^*/\|\phi^*\|$ in the above to conclude that $\phi^\rd$ is a complete contraction (note that $\phi^\rd$ is unique). 

\smnoind
(b) If $\phi, \psi\in \Mor(S,T)$ satisfying $\phi^\rd = \psi^\rd$, then $\phi^* = \psi^*$, which gives $\phi = \psi$. 
For the surjectivity, we consider $\chi\in \Morw(T^\rd, S^\rd)$. 
The weak-$^*$-continuity of $\chi$ produces, via Proposition \ref{prop:reg}(c), a bounded linear map $\phi:S\to T$ with
\begin{equation}\label{eqt:chi-phi-*}
\chi \circ \mu_{T^*} = \mu_{S^*}\circ \phi^*.
\end{equation}
Moreover, as $\chi\in \Mor(T^\rd,S^\rd)$, we obtain from Proposition \ref{prop:predual}(b) and part (a) above, a map $\chi^\rd\in \Morw(T^{\rd\rd}, S^{\rd\rd})$ with $\chi^\rd\circ \mu_{(S^\rd)^*} = \mu_{(T^\rd)^*}\circ \chi^*$. 
It follows that 
$$\chi^\rd \circ \tau_S = \tau_T\circ \phi^{**}.$$ 
From this, and Lemma \ref{lem:quasi-oper-sys}(a), one knows that $\phi^{**}\in \Mor(S^{**}, T^{**})$. 
Now, Lemma \ref{lem:bidual} tells us that $\phi\in \Mor(S,T)$, and the uniqueness in part (a) implies that $\chi = \phi^\rd$. 

\smnoind
(c) If $\phi:S\to T$ is a MOS isomorphism, then $\phi^*$ is a weak-$^*$-homeomorphic MOS isomorphism, and so is $\phi^\rd$. 
For the second statement, we know from part (b) above that $\chi = \phi^\rd$ for a unique $\phi\in \Mor(S,T)$. 
Moreover, the map $\chi^\rd\in \Morw(T^{\rd\rd}, S^{\rd\rd})$ is again a weak-$^*$-homeomorphic MOS isomorphism. 
Therefore, it follows from $\chi^\rd \circ \tau_S = \tau_T\circ \phi^{**}$ and Lemma \ref{lem:quasi-oper-sys}(a) that $\phi^{**}$ is a weak-$^*$-homeomorphic MOS isomorphism. 
Now, Lemma \ref{lem:bidual} implies that $\phi$ is a MOS isomorphism.

\smnoind
(d) Let $D_\#$ be the predual of $D$. 
We consider $j: D_\# \to D^\rd$ to be the composition of the canonical map from $D_\#$ to $D^*$ with  $\mu_{D^*}:D^*\to D^\rd$. 
By parts (a) and (b) of Proposition \ref{prop:reg} as well as the assumption on $D$, we know that $\mu_{D^*}$ is a MOS isomorphism, and hence $j$ is a MOS isomorphism from $D_\#$ onto the operator subsystem $S:=j(D_\#)$. 
Since  $j^*:S^*\to (D_\#)^* = D$ (see Corollary \ref{cor:reform-CP}) is a weak-$^*$-homeomorphic MOS isomorphism, $S^*$ is a dual quasi-operator system, and so, $S$ is dualizable. 
Therefore, $\mu_{S^*}:S^* \to S^\rd$ is a weak-$^*$-homeomorphic MOS isomorphism (see Proposition \ref{prop:reg}), and the composition of $j^*$ with $\mu_{S^*}^{-1}$ will then be a weak-$^*$-homeomorphic MOS isomorphism from $S^\rd$ to $D$. 

\smnoind
(e) As in the proof of part (b) above, there exists $\phi\in \Mor(S,T)$ with $\chi = \phi^\rd$ and $\chi^\rd \circ \tau_S = \tau_T\circ \phi^{**}$. 
Since $\chi$ is a complete contraction, we know from part (a) that $\chi^\rd$ is a complete contraction. 
It follows that $\tau_T\circ \kappa_T\circ \phi = \chi^\rd\circ \tau_S \circ \kappa_S$ is a complete contraction, because $\tau_S\circ \kappa_S$ is completely isometric.
Hence, $\phi$ is a complete contraction (as $\tau_T\circ \kappa_T$ is completely isometric). 

When $\chi$ is a completely isometric MOS isomorphism, we know from Lemma \ref{lem:image-weak-st-cont-bdd-below} that $\chi$ is a weak-$^*$-homeomorphism. 
Thus, $\chi^{-1}$ exists and belongs to $\Morc_w(S^\rd;T^\rd)$.
By the first statement, one can find $\psi\in \Morc(T,S)$ with $\chi^{-1} = \psi^\rd$. 
It follows that $\psi^\rd = (\phi^\rd)^{-1}$, and hence $\psi = \phi^{-1}$ (as both $\mu_{S^*}$ and $\mu_{T^*}$ are bijective). 
Consequently, $\phi$ is a completely isometric MOS isomorphism. 
\end{proof}

\medskip

Parts (a) and (b) of the above tell us that there is a full and  faithful contravariant functor from the category of dualizable operator systems to the category of dualizable dual operator systems. 
Parts (c) and (d) state that this contravariant functor induces an equivalence of categories if one identifies two operator systems that are MOS isomorphic, and identifies two dual operator systems that are weak-$^*$-homeomorphically MOS isomorphic. 

\medskip

Part (e) above tells us that when restricted to the subcategory of dualizable operator systems $S$ with $\tau_S\circ \kappa_S$ being a complete isometry, the above functor is injective on objects. 
Moreover, in this case, if one uses  completely positive complete contractions and weak-$^*$-continuous completely positive complete contractions as the respective morphisms, the above functor is again a full functor. 

%end v2
\medskip

In the following, we will show that the collection $\CS$ of dualizable operator systems $S$ with  $\tau_S$ being completely isometric (which implies $\tau_S\circ \kappa_S$ being completely isometric, because of Lemma \ref{lem:bidual}) includes all $C^*$-algebras and all unital \textcolor{magenta}{complete} operator systems. 

\medskip

\begin{thm}\label{thm:dual-oper-sys}
Let $S$ be a (not necessarily unital) $C^*$-algebra or a unital \textcolor{magenta}{complete}  operator system. 

\smnoind
(a) Then $S$ is dualizable, $\|f\|^\rd \leq \|f\|_{M_\infty(S^*)} \leq 4 \|f\|^\rd$ ($f\in M_\infty(S^*)$) and 
\begin{equation}\label{eqt:norm-dual-plus}
	\|g\|_{M_\infty(S^*)} = \|g\|^\rd \qquad (g\in M_\infty(S^*)^+). 
\end{equation}

\smnoind
(b) The canonical bijection $\tau_S:S^{**} \to (S^\rd)^\rd$ is a complete isometry. 
\end{thm}
\begin{proof}
(a) By Proposition \ref{prop:reg}(a), one has $\|f\|^\rd \leq \|f\|_{M_\infty(S^*)}$ (recall that  $\|\cdot\|^\rd$ is the matrix norm induced by $\mu_{S^*}$). 
Under the assumption on $S$, elements in $B_{M_n(S)^\sa}$ are of the form $u-v$ for some $u,v\in B^+_{M_n(S)}$ ($n\in \BN$).
%If $S$ is a unital operator system, then for each $n\in \BN$, for every elements in $x\in B_{M_n(S)^\sa}$, one has $x = (1+x)/2 - (1-x)/2$ and $(1\pm x)/2\in B^+_{M_n(S)}$.
This shows that $S$ is dualizable. 
Moreover, it follows from  the proof of $(3)\Rightarrow (4)$ of Theorem \ref{thm:quasi-oper-sys}(a) that $\|f\|_{M_m(S^*)}\leq 4 \|\mu_{S^*}^{(m)}(f)\|= 4 \|f\|^\rd$ ($m\in \BN$, $f\in M_m(S^*)$). 

Next, we will verify Relation \eqref{eqt:norm-dual-plus}.
In fact, 
%for any $n\in \BN$ and $\varphi\in \WQ^{S^*}_n$, we can find $u\in M_n(S^{**})^+\cap B_{M_n(S^{**})}$ with $\varphi = \theta_u$, and it follows from Lemma \ref{lem:bidual} as well as Lemma \ref{lem:weak-st-cont} that $u\in M_n(S)^+\cap B_{M_n(S)}$. 
%Pick 
consider any $m\in \BN$ and $g\in M_m(S^*)^+$. 
The completely bounded map $\theta_g:S \to M_m(\BC)$ is completely positive.
If either $S$ is a $C^*$-algebra or a unital operator system, then there is a net   $\{a_{i}\}_{i\in \KI}$ in $S^+\cap B_S$ such that
$$\sup \{\|\theta_g^{(n)}(a_{i}\otimes 1_n)\|: i\in \KI\} = \|\theta_g^{(n)}\| \qquad (n\in \BN).$$
Thus, 
$\|g\|^\rd = \|\mu_{S^*}^{(m)}(g)\| \geq \sup_{i\in \KI; n\in \BN} \|\theta_g^{(n)}(a_{i}\otimes 1_n)\| = \sup_{n\in \BN} \|\theta_g^{(n)}\| = \|g\|_{M_m(S^*)}.$

\smnoind
(b) Consider $k\in \BN$ and $x\in M_k((S^\rd)^*)$. 
We know from Relation \eqref{eqt:norm-dual-plus} and Proposition  \ref{prop:reg}(a) that $\mu_{S^*}^{(n)}\big(M_n(S^*)^+\cap B_{M_n(S^*)}\big) = M_n(S^\rd)^+\cap B_{M_n(S^\rd)}$ ($n\in\BN$). 
From this, as well as Relations \eqref{eqt:bidual-norm} and \eqref{eqt:norm-bidual}, we have $\big\|\mu_{(S^\rd)^*}^{(k)}(x)\big\|_{M_k((S^\rd)^\rd)} = \big\|(\mu_{S^*}^*)^{(k)}(x)\big\|_{M_k(S^{**})}$, as required. 
\end{proof}

\medskip

Let $S$ be \textcolor{magenta}{a complete} operator system with the bidual operator system $S^{**}$ being unital. 
By  Theorem \ref{thm:dual-oper-sys}(a), we know that that $S^{**}$ is dualizable. Using Lemma \ref{lem:bidual} and Corollary \ref{cor:bidual-quasi-os}(b),  one can then show that $S$ itself is dualizable. 
In \cite[Theorem 4.2]{Huang}, a characterization of an operator system $S$ with $S^{**}$ being unital is given, in term of the existence of a ``normed approximate order unit'' on $S$.
We wonder if it is possible to use the materials in \cite{Huang} to generalize Theorem \ref{thm:dual-oper-sys}(b) to all operator systems whose bidual are unital. 
Since we do not need this general statement, we will leave it to the readers.

\medskip

The following tells us that, under the duality funtor as in the above, the category consisting of all $C^*$-algebras and all unital \textcolor{magenta}{complete} operator systems, with morphisms being completely positive complete contractions, is a full subcategory of the category of dual operator systems, with morphisms being weak-$^*$-continuous completely positive complete contractions.

\medskip

\begin{cor}\label{cor:C-st-or-unital}
Let $S$ and $T$ be either a $C^*$-algebra or a unital \textcolor{magenta}{complete}  operator system. 

\smnoind
(a) For any $\phi\in \Morc(S,T)$, there is  a unique $\phi^\rd\in \Ccpw(T^\rd, S^\rd)$ satisfying $\phi^\rd\circ \mu_{T^*} = \mu_{S^*}\circ \phi^*$. 

\smnoind
(b) The assignment $\phi\mapsto \phi^\rd$ is a bijection from  $\Morc(S,T)$ onto $\Ccpw(T^\rd, S^\rd)$. 

\smnoind
(c) There is a completely isometric MOS isomorphism from $S$ onto $T$ if and only if there is a weak-$^*$-continuous completely isometric MOS isomorphism from $T^\rd$ onto $S^\rd$. 
\end{cor}
\begin{proof}
	By Theorem \ref{thm:dual-oper-sys}, $S$ and $T$ are dualizable and both $\tau_S\circ \kappa_S$ and $\tau_T\circ \kappa_T$ are complete isometry (see also Lemma \ref{lem:bidual}). 
	
	\smnoind
	(a) This follows from part (a) of Theorem \ref{thm:dualization-functor}. 
	
	\smnoind
	(b) This part follows from parts (a), (b) and (e) of Theorem \ref{thm:dualization-functor}. 
	
	\smnoind
	(c) Suppose that $\phi: S \to T$ is a completely isometric MOS isomorphism. 
	Then $\phi^{-1}$ exists and belongs to $\Morc(T;S)$. 
	Since the two maps $\phi^\rd\in \Morc_w(T^\rd; S^\rd)$ and  $(\phi^{-1})^\rd\in \Morc_w(S^\rd;T^\rd)$ are inverse of each other, we know that $\phi^\rd$ is a weak-$^*$-homeomorphic completely isometric MOS isomorphism. 
	The gives the forward implication. 
	The backward implication is precisely the second statement of Theorem \ref{thm:dualization-functor}(e). 
\end{proof}

\medskip

\begin{eg}\label{eg:dualizable}
Let $N\in \BN$ and $S$ be the $N$-dimensional commutative $C^*$-algebra $\ell_\infty^N$. 
We claim that $S^\rd$ is not a unital operator system. 
In fact, suppose on the contrary that there exists an order unit $u = (u_1,\dots,u_N)\in (S^\rd)^+ \subseteq \ell^N_1$ with $\|u\|^\rd = 1$ that defines the matrix norm on $S^\rd$. 
By Theorem \ref{thm:dual-oper-sys}(a), we have $B_{S^*}\subseteq B_{S^\rd}$.
For $k =1,...,N$, if $e_k\in S^* = \ell_1^N$ is the element with the $k$-entry being $1$ and all the other entries being zero, then $0 \leq e_k\in B_{S^*}\subseteq B_{S^\rd}$, and hence $e_k\leq u$; i.e., $u_k\geq 1$.  
This implies that $(1, \dots, 1)\leq u$.  
However, in this case, one has  $\|u\|_{S^*} \geq N \gneq \|u\|^\rd$, which contradicts Relation \eqref{eqt:norm-dual-plus}.  
\end{eg}

\medskip

Let us recall that a unital operator system $T\subseteq \CL(\KH)$ is a \emph{graph system} if it is closed under the weak-$^*$-topology (see \cite{Weaver, Yashin}).  
The following is a direct consequence of Theorem \ref{thm:dual-oper-sys}(a) as well as Corollaries \ref{cor:reform-CP} and \ref{cor:bidual-quasi-os}(a). 

\medskip

%v2
\begin{cor}
Let $T\subseteq \CL(\KH)$ be a graph system and $T_\#:= \{\omega|_T: \omega\in \CL(\KH)_* \}\subseteq T^*$. 
Then the canonical map from $T$ to $(T_\#)^*$ is a weak-$^*$-homeomorphic completely order monomorphic complete isometry. 
Moreover, $T_\#$ admits a largest matrix norm $\|\cdot\|_\rd$ that is dominated by and equivalent to the restriction of the dual matrix norm $\|\cdot\|^*$ on $T_\#$, under which $T_\#$ becomes a dualizable operator system.  
In this case, we have $\|x\|_\rd \leq \|x\|^* \leq 4 \|x\|_\rd$ ($x\in M_\infty(T_\#)$) and $\|u\|_\rd = \|u\|^*$ $(u\in M_\infty(T_\#)^+)$. 
\end{cor}

\medskip

The resulting operator system in the above can be regarded as a ``predual graph system''. 

\medskip

\appendix\section{Two remarks}

\medskip

\subsection{Alternative dual objects for unital operator systems.}
In this subsection, we present an alternative dual object for a unital operator system, in the form of a unital dual operator system. 

\medskip

Let us call a unital complete contraction from a unital operator system $S$ to  $\BC$ 
a \emph{character} on $S$.
When a unital operator system is equipped with a fixed character, we call it a \emph{pointed unital operator system}. 
For example, if $K$ is a compact Hausdorff space and $x\in K$ is a fixed point, then the evaluation $\delta_x$ at $x$ is a character on $C(K)$, and $(C(K), \delta_x)$ is a pointed unital operator system. 

\medskip

By a \emph{pointed unital dual operator system}, we mean a unital dual operator system equipped with a ``weak-$^*$-continuous character''. 
Let $W$  and $V$ be unital dual operator systems. 
The set of weak-$^*$-continuous unital complete contractions from $W$ to $V$ will be denoted by $\Ucpw(W,V)$. 
If $\omega\in \Ucpw(W,\BC)$ and $\delta\in \Ucpw(V,\BC)$, we let $\Ucp_w^{\omega,\delta}(W,V)$ be the subset consisting of ``character preserving'' maps; i.e., 
$$\Ucp_w^{\omega,\delta}(W,V) := \{\Phi\in \Ucpw(W,V): \omega = \delta\circ \Phi \}.$$ 

\medskip

In the following, we outline a representation $\bar \mu_W$ for a unital dual operator system $W$ (which is a weak-$^*$-version of the corresponding construction in \cite[Remark 2.11]{HN}), similar to the representation $\mu_D$ of a dual SMOS $D$. 
For $n\in \BN$, we set $\mathcal{WS}^W_n := \Ucpw(W; M_n(\BC))$. 
By a similar argument as that of Proposition  \ref{prop:reg}(a),  the canonical map 
$$\bar \mu_W: W \to {\bigoplus}_{n\in \BN}^{\ell_\infty} \ell_\infty(\mathcal{WS}^W_n; M_n(\BC))$$ is a weak-$^*$-continuous unital complete contraction. 
Moreover, as $W$ is a unital self-adjoint weak-$^*$-closed subspace of some $\CL(\KH)$, it follows from Relation \eqref{eqt:norm-L(H)} that $\bar \mu_W$ is complete isometric, and hence it is also completely order monomorphic. 
Furthermore, Lemma \ref{lem:image-weak-st-cont-bdd-below} tells us that $\bar \mu_W$ is a weak-$^*$-homeomorphism from $W$ onto $\bar \mu_W(W)$. 

\medskip

On the other hand, for a dual MOS $D$, we equip $D^\wu:= \overline{\mu_D(D) + 1_{N(D)}}^{\sigma(N(D),N(D)_*)}$ with the induced dual operator system structure from $N(D)$. 

\medskip

\begin{prop}\label{prop:alt-def-dual-MOS}
Let $D$ be a dual MOS.   
Then $\big(D^\wu, \mu_D\big)$ is the \emph{weak-$^*$-partial unitization} for $D$ in the following sense:
for any unital dual operator system $W$ and any $\Phi\in \Ccpw(D,W)$, there is a unique map $\ti \Phi\in \Ucpw\big(D^\wu; W\big)$ 
satisfying $\Phi = \ti \Phi
\circ \mu_D$. 
\end{prop}

\medskip

In fact, by replacing $W$ with the weak-$^*$-closure of $\Phi(D) + \BC 1_W$ if necessary, we may assume that $\Phi(D) + \BC 1_W$ is weak-$^*$-dense in $W$. 
By substituting $\mu_V$ and $\WQ^V_n$  with $\bar \mu_W$ and $\WS^W_n$, respectively, in the argument for Proposition \ref{prop:reg}(d), one obtains the required map $\ti \Phi$ in the above.

\medskip

Let $S$ and $T$ be unital \textcolor{magenta}{complete} operator systems. 
We set $S^\diamond$ to be the weak-$^*$-partial unitization of $S^*$; i.e.,  
$$S^\diamond := S^\rd + \BC 1_{N(S^*)}$$ 
(see Proposition \ref{prop:reg}(c)).
The zero map from $S^*$ to $\BC$ induced, via Proposition \ref{prop:alt-def-dual-MOS}, a weak-$^*$-continuous character $\omega_S$ on $S^\diamond$ with its kernel being $S^\rd$. 
This gives a pointed unital dual operator system $(S^\diamond, \omega_S)$. 

\medskip

For $\chi\in \Ccpw(T^\rd, S^\rd)$, it follows from  $\chi\circ \mu_{T^*}\in  \Ccpw(T^*, S^\diamond)$ that there is $\check \chi \in \Ucp_w(T^\diamond, S^\diamond)$ with $\check \chi \circ \mu_{T^*} = \chi\circ \mu_{T^*}$ (see Proposition  \ref{prop:alt-def-dual-MOS}).
Obviously, $\check \chi$ is the unique unital extension of $\chi$ and is character preserving. 
Conversely, any map in $\Ucp_w^{\omega_T,\omega_S}(T^\diamond, S^\diamond)$ will restrict to an element in $\Ccpw(T^\rd, S^\rd)$ whose unital extension coincides with the original map. 
Thus, we obtain a bijection from $\Ccpw(T^\rd, S^\rd)$ onto $\Ucp_w^{\omega_T,\omega_S}(T^\diamond, S^\diamond)$.

\medskip
 
Now, Corollary \ref{cor:C-st-or-unital} gives the following result (for part (c), we note that a completely isometric MOS isomorphism between unital operator systems will automatically preserve the order units).
This result means that ``the duality construction'' induces a full and faithful functor, which is also  is injective on objects, from the category of unital \textcolor{magenta}{complete} operator systems, with completely positive complete contractions as their morphisms, to the category of pointed unital dual operator systems, with character preserving weak-$^*$-continuous unital complete contractions as their morphisms.  
In plain words, the former category is a full category of the latter.

\medskip

\begin{cor}\label{cor:unital-dual}
	Let $S$ and $T$ be unital \textcolor{magenta}{complete} operator systems. 
	
	\smnoind
	(a) For any $\phi\in \Morc(S,T)$, there is  a unique $\phi^\diamond\in \Ucpw(T^\diamond, S^\diamond)$ satisfying $\phi^\diamond\circ \mu_{T^*} = \mu_{S^*}\circ \phi^*$. 
	
	\smnoind
	(b) The assignment $\phi\mapsto \phi^\diamond$ is a bijection from  $\Morc(S,T)$ onto $\Ucp_w^{\omega_T,\omega_S}(T^\diamond, S^\diamond)$. 
	
	\smnoind
	(c) There is a character preserving weak-$^*$-continuous unital complete isometry from $T^\diamond$ onto $S^\diamond$ if and only if there is a unital complete isometry from $S$ onto $T$. 
\end{cor}

\medskip

Note that although the unital dual operator system $S^\diamond$ is not MOS isomorphic to the dual space of $S$, the dual operator subsystem $\ker \omega_S$ is weak-$^*$-homeomorphically MOS isomorphic to $S^*$. 

\medskip

Obviously, if we use unital complete contractions as morphisms between unital operator systems, then the above duality functor is still faithful, but will not be full.

\medskip

\subsection{Corresponding results for ordered normed spaces and function systems.}\label{subsec:non-matrix} 
SMOS in the main context can be regarded as the matrix analogues of ordered normed spaces (while MOS are analogues of ordered normed spaces with proper cones), and operator systems can be regarded as matrix analogues of function systems. 
More precisely, we call a self-adjoint subspace of $C(K)$ for some compact Hausdroff space $K$, equipped with the induced ordered normed space structure from $C(K)$, a \emph{function system}. 
A function system is said to be \emph{unital} if it contains the identity of $C(K)$ (notice that the term ``function systems'' in \cite{BM} are unital function systems in our sense). 
On the other hand, a \emph{quasi-function system} is an ordered normed space such that there is an ordered normed space isomorphism (not necessarily isometric) from this space to a function system. 

\medskip

Furthermore, a \emph{dual ordered Banach space} is the dual space $X^*$ of an ordered \textcolor{magenta}{Banach} space $X$,  equipped with the dual norm and the dual cone.
A \emph{dual function system} is a weak-$^*$-closed self-adjoint subspace of $\ell_\infty(I)$ for a set $I$, equipped with the induced weak-$^*$-topology and the induced ordered Banach space structure. 
By \cite[Corollary 2.2]{BM}, dual function systems can also be described as weak-$^*$-closed self-adjoint subspaces of $C(K)^{**}$ for some compact Hausdorff space $K$. 
If $X$ is a dual ordered \textcolor{magenta}{Banach} space such that there is a weak-$^*$-continuous ordered Banach space isomorphism (not necessarily isometric) to a dual function system, then $X$ is called a \emph{dual quasi-function system}.

\medskip

Most of the results in the main context of this article have analogues in the case of ordered \textcolor{magenta}{Banach} spaces and function systems. 
The arguments are similar to, and easier than, the matrix case results (in particular, the technical consideration as that of Lemma \ref{lem-BM1.1} is no longer required), but one needs to replace $M_n(X)$ and $\CL(\KH)$ with $X$ and $\ell^\infty(I)$, respectively.
Note that when considering the bidual  of $C(K)$ in some of the proofs, one will need \cite[Corollary 2.2]{BM}. 
In particular, for the argument for the analogue of Lemma \ref{lem:norm-pos-cont}, one requires a weak-$^*$-continuous positive isometry from $C(K)^{**}$ to some $\ell^\infty(I)$. 

\medskip

Some of these analogues seem to be well-known; e.g.,  for an ordered \textcolor{magenta}{Banach} space $X$, 
\begin{itemize}
	\item $X^*$ has a proper cone if and only if the linear span of $X^+$ is norm dense in $X$;
	\item $X^*$ is a dual quasi-function system if and only if $X$ satisfies a form of bounded decomposition property (namely, $B_X\cap X^+ - B_X\cap X^+$ is a zero neighborhood of $X^\sa$); 

	\item if $X$ is a quasi-function system (respectively, function system), then $X^{**}$ is a dual quasi-function system (respectively, dual function system). 
\end{itemize}

\medskip

On the other hand, some of these analogues seem to be new; e.g. for an ordered \textcolor{magenta}{Banach} space $X$, if we denote by $\nu_{X^*}: X^*\to \ell_\infty(B_X\cap X^+)$ the evaluation map, and by $X^\flat$ the dual function system $\overline{\nu_{X^*}(X^*)}^{\sigma(\ell_\infty(B_X\cap X^+), \ell_1(B_X\cap X^+))}$, then 

\begin{itemize}
	\item $X^*$ is a dual quasi-function system if and only if $\nu_{X^*}(X^*)$ is norm-closed  in $\ell_\infty(B_X\cap X^+)$ (in this case, it is also weak-$^*$-closed  in $\ell_\infty(B_X\cap X^+)$); 
	
	\item $X^*$ is a dual function system if and only if $\nu_{X^*}$ is isometric;
	
	\item when $F$ is a dual function system and $\phi\in \Pcw(X^*;F)$, there is a unique $\bar \phi\in \Pcw(X^\flat;F)$ with $\bar \phi\circ \nu_{X^*} = \phi$, where $\Pcw$ is the set of weak-$^*$-continuous positive contraction;

	\item when $X^*$ is a dual quasi-function system, the map $\nu_{X^*}$ is a weak-$^*$-homeomorphic ordered Banach space isomorphism from $X^*$ onto $X^\flat$,  and $(X^\flat)^*$ is again a dual quasi-function system. 
	
\end{itemize}

\medskip
Furthermore,

\begin{itemize}

	\item if $D$ is a dual function system, then $D^*$ is a quasi-function system if and only if there is a quasi-function system $Z$ with $D\cong Z^\flat$, as dual ordered Banach spaces. 

\begin{comment}	
	\item if $X$ and $Y$ are quasi-function systems with $X^*$ and $Y^*$ being dual quasi-function systems, then 
	
	\begin{enumerate}
		\item for any bounded positive map $\phi:X\to Y$, there is a unique weak-$^*$-continuous bounded positive map $\phi^\flat: Y^\flat\to X^\flat$ with $\phi^\flat\circ \nu_{Y^*} = \nu_{X^*}\circ \phi^*$;
		
		\item $\phi\mapsto \phi^\flat$ is a bijection between the set of bounded positive map from $X$ to $Y$ and the set of weak-$^*$-continuous bounded positive map from $Y^\flat$ to $X^\flat$; 
		
		\item $X^\flat \cong Y^\flat$ as dual ordered Banach spaces if and only if $X\cong Y$ as ordered \textcolor{magenta}{Banach} spaces; 
		
	\end{enumerate}
\end{comment}	
	\item if $X$ and $Y$ are either a unital \textcolor{magenta}{complete} function system or a space of the form $\{f\in C(K): f(x_0) = 0 \}$ for a fixed element $x_0\in K$, then 
	\begin{enumerate}[(i)]
		
		\item $X^*$ and $Y^*$ are dual quasi-function systems;
		
		\item the canonical map from $X^{**}$ to $(X^\flat)^\flat$ is isometric; 
		\item for any $\psi\in \Pc(X,Y)$, there is a unique $\psi^\flat\in \Pcw(Y^\flat, X^\flat)$ with $\psi^\flat\circ \nu_{Y^*} = \nu_{X^*}\circ \psi^*$, where $\Pc$ is the set of positive contractions and the set of weak-$^*$-continuous positive contractions, respectively;
		\item $\psi\mapsto \psi^\flat$ is a bijection from $\Pc(X,Y)$ onto $\Pcw(Y^\flat, X^\flat)$;
		\item there is an isometric order isomorphism from $X$ to $Y$ if and only if there is a weak-$^*$-homeomorphic isometric order isomorphism from $Y^\flat$ to $X^\flat$. 
	\end{enumerate}
\end{itemize}

\medskip

\section*{Acknowledgement}

The author is supported by National Natural Science Foundation of China (11871285). 
He would like to thank Prof. van Suijlekom for informing him of \cite{CvD} as well as  \cite{Far}, and to thank Prof. Han for informing him of \cite{Han}.
Moreover, he would  like to express his gratitude to the referee for some comments leading to improvement of this paper.


\medskip

\begin{thebibliography}{99}

\bibitem{AE}
L. Asimow and A. J. Ellis, \emph{Convexity theory and its applications in functional analysis}, London Math. Soc. Mono. \textbf{16}, Academic Press, London (1980).


\bibitem{BM} D. Blecher and B. Magajna, Dual Operator systems, Bull. Lond. Math. Soc. \textbf{43} (2011), 311-320. 

\bibitem{Bon}
F.F. Bonsall, Endomorphisms of a partially ordered vector space without order unit, J. Lond. Math. Soc. \textbf{30} (1955), 144-153. 

\bibitem{CJ}
R. Carbone and A. Jen\v{c}ov\'{a}, On period, cycles and fixed points of a quantum channel, Ann. Henri Poin. \textbf{21} (2020), 155-188. 


\bibitem{CE}
M-D Choi and E. Effros, Injectivity and operator spaces. J. Funct. Anal., \textbf{24}
(1977), 156-209.

\bibitem{CvD}
A. Connes and W.D. van Suijlekom, Spectral truncations in noncommutative geometry and operator systems, Comm. Math. Phys. \textbf{383} (2021), 2021-2067. 

\bibitem{DSW} 
R. Duan, S. Severini and A. Winter, Zero-error communication via quantum channels, non-commutative graphs, and a quantum Lov\'{a}sz number, IEEE Transactions on Information Theory, \textbf{59} (2013), 1164-1174.

\bibitem{ER} E.G. Effros and Z.-J. Ruan, \emph{Operator spaces}, London Math. Soc. Mono. New Series 23, Oxford Univ. Press, (2000).

\bibitem{Far}
D. Farenick, The operator system of Toeplitz matrices, preprint (arXiv: 2103.16546v3). 

\bibitem{GMS}
V.P. Gupta, P. Mandayam and V.S. Sunder, \emph{The Functional Analysis of Quantum Information Theory} (A Collection of Notes Based on Lectures by Gilles Pisier, K. R. Parthasarathy, Vern Paulsen and Andreas Winter), Lecture Notes in Physics \textbf{902}, Springer, (2015). 

\bibitem{Han}
K.H. Han, Matrix regular operator space and operator system, J. Math. Anal. Appl. \textbf{367} (2010), 516-521. 

\bibitem{Huang}
X.J. Huang, Approximately unital operator systems, Sci. China Math. \textbf{54} (2011), 1243-1257.

\bibitem{HN}
X.J. Huang and C.K. Ng,  An abstract characterization of unital operator spaces, J. Operator Theory \textbf{67} (2012), 289-298.


\bibitem{KPTT}
A. Kavruk, V.I. Paulsen, I.G. Todorov and M. Tomforde, Tensor products of operator systems, J. Funct. Anal. \textbf{261} (2011), 267-299.

\bibitem{LN}
J.Z. Li and C.K. Ng, Tensor products for non-unital operator systems, J. Math. Anal. Appl. \textbf{396} (2012), 601-605.

\bibitem{Ng-reg-mod} 
C.K. Ng, Regular normed bimodules, J. Oper. Theory, \textbf{56} (2006), 343-355

%\bibitem{Ng-dual}
%C.K. Ng, On unital duals of unital  operator systems, in preparation. 

\bibitem{Ng-MOS}  C.K. Ng, Operator subspaces of $\CL(H)$ with induced matrix orderings, Indiana Univ. Math. J. \textbf{60} (2011), 577-610.

\bibitem{Paul}
V.I. Paulsen, Completely bounded maps on $C^*$-algebras and invariant operator ranges, Proc. Amer. Math. Soc. \textbf{86} (1982), 91-96. 

\bibitem{Weaver}
N, Weaver, Quantum graphs as quantum relations, preprint (arXiv:1506.03892). 

\bibitem{Wern-subsp} W. Werner, Subspaces of $L(H)$ that are $*$-invariant, J. Funct. Anal. 193 (2002), 207-223.

\bibitem{Wern1} W. Werner, Multipliers on matrix ordered operator spaces and some $K$-groups, J. Funct. Anal. \textbf{206} (2004), 356-378.

\bibitem{Yashin}
V.I. Yashin,  Properties of operator systems, corresponding to channels, Quantum Inf. Process \textbf{19} (2020), 195.
\end{thebibliography}
\end{document}